\documentclass[reqno,final]{amsart}

\usepackage{graphicx}
\usepackage{color,multirow}
\usepackage{amsmath,amsthm,amssymb,amsfonts}
\usepackage[foot]{amsaddr}
\usepackage{mathrsfs} % for math script font
\usepackage{url}
\usepackage{geometry} % for adjusting margins
\usepackage{enumerate} % for easy roman numerals in lists
\usepackage[color,notref,notcite]{showkeys} % for displaying labels in the margins
\usepackage[normalem]{ulem} % for corrections
\usepackage{bm} % for bold symbols

\theoremstyle{plain}
\newtheorem{theorem}{Theorem}[section]
\newtheorem{lemma}[theorem]{Lemma}

\newtheorem{corollary}[theorem]{Corollary}

\theoremstyle{definition}
\newtheorem{definition}[theorem]{Definition}

\numberwithin{equation}{section}

\def\bhyp#1{\begin{equation}\label{#1}\begin{array}{c}}
\def\ehyp{\end{array}\end{equation}}

\theoremstyle{remark}
\newtheorem{remark}{Remark}[section]

\definecolor{labelkey}{rgb}{0.6,0,1}

\newcommand{\cT}{\mathcal{T}}
\newcommand{\cF}{\mathcal{F}}
\newcommand{\cS}{\mathcal{S}}
\newcommand{\fS}{\mathfrak{S}}
\newcommand{\cD}{\mathcal{D}}
\newcommand{\cI}{\mathcal{I}}
\newcommand{\distt}{\mathcal{D'}(\Omega\times(0,T))} % on Omega x (0, T)

\newcommand{\RR}{{\mathbb R}}
\newcommand{\NN}{{\mathbb N}}

\renewcommand{\O}{\Omega}

\newcommand{\K}{\mathbf{K}}

\newcommand{\U}{\mathbf{u}}
\newcommand{\IdM}{\mathbf{I}}
\newcommand{\E}{\mathbf{E}}
\newcommand{\normvec}{\mathbf{n}}

\newcommand{\bvarphi}{\boldsymbol{\varphi}}
\newcommand{\A}{\mathbf{A}}
\newcommand{\one}{\mathbf{1}}

\def\XXint#1#2#3{{\setbox0=\hbox{$#1{#2#3}{\int}$ }
\vcenter{\hbox{$#2#3$ }}\kern-.6\wd0}}

\newcommand{\trunc}{{\mathbb{T}}} % truncation
\newcommand{\dive}{\operatorname{div}} % divergence
\newcommand{\ud}{\, \mathrm{d}} % roman script d in integrals
\newcommand{\norm}[1]{\left\lVert#1\right\rVert} % norms

\newcommand{\Id}{\operatorname{Id}} % identity
\newcommand{\ol}[1]{\overline{#1}}

\newcommand{\weakto}{\rightharpoonup}

\newcommand{\weak}{\mbox{-}\mathrm{w}}

\newcommand{\dm}{d_m}
\newcommand{\dl}{d_l}
\newcommand{\dt}{d_t}

\newcommand{\Leb}[3][ ]{L^{#2}(0, T; L^{#3}(\O)^{#1})} % for vector-valued Lebesgue spaces
\newcommand{\D}[3][ ]{\mathbf{D}_{#1}(#2, #3)}
\newcommand{\Dsq}[3][ ]{\mathbf{D}_{#1}^{1/2}(#2, #3)}
\newcommand{\DD}{\mathbf{D}}

\newcommand{\Ubar}{\overline{\U}}
\newcommand{\cbar}{\overline{c}}
\newcommand{\pbar}{\overline{p}}

\newcommand{\disc}{\cD}
\newcommand{\discm}{\cD_{m}}

\newcommand{\unknowns}{X_\disc}

\newcommand{\deltat}{\delta t^{(n+\frac{1}{2})}}
\newcommand{\deltatm}{\delta t_m^{(n+\frac{1}{2})}}
\newcommand{\deltatD}{\delta t_\disc}
\newcommand{\PiD}{\Pi_{\disc}}
\newcommand{\PiDe}{\Pi_{\disc}^{\mathrm{ex}}}

\newcommand{\PiDm}{\Pi_{\disc_m}}
\newcommand{\PiDme}{\Pi_{\disc_m}^{\mathrm{ex}}}

\newcommand{\gradD}{\nabla_{\disc}}

\newcommand{\gradDm}{\nabla_{\disc_m}}
\newcommand{\gradDme}{\nabla_{\disc_m}^{\mathrm{ex}}}

\newcommand{\Udisc}{\U_\disc}
\newcommand{\deltaD}{\delta_{\disc}}

\newcommand{\interpD}{\cI_\disc}
\newcommand{\unknownsm}{X_{\disc_m}}
\newcommand{\interpDm}{\cI_{\disc_m}}
\newcommand{\deltatDm}{\delta t_{\disc_m}}
\newcommand{\pDm}{p_{\discm}}
\newcommand{\cDm}{c_{\discm}}
\newcommand{\UDm}{\U_{\discm}}
\newcommand{\deltaDm}{\delta_{\disc_m}}
\newcommand{\Tm}{T_{m}}
\newcommand{\km}{k_{m}}
\newcommand{\ellip}{\mathrm{ell}}
\newcommand{\para}{\mathrm{para}}

\begin{document}
% ------------------------- TITLE INFO -------------------------

\title[]{Unified convergence analysis of numerical schemes for a miscible displacement problem}

% authors
\author[J. Droniou]{J\'er\^ome Droniou$^1$}
	\address{$^{1,4}$School of Mathematical Sciences\\
		Monash University\\
		Clayton, Victoria 3800 Australia}
	\email{jerome.droniou@monash.edu}
	\thanks{$^1$Corresponding author. Tel: (+61) 3 9905 4489. Fax: (+61) 3 9905 4403}
\author[R. Eymard]{Robert Eymard$^2$}
	\address{$^{2,3}$Universit\'e Paris-Est\\
		Laboratoire d'Analyse et de Math\'ematiques Appliqu\'ees\\
		UMR 8050\\
		5 boulevard Descartes\\
		Champs-sur-Marne 77454 Marne-la-Vall\'ee Cedex 2\\
		France}
	\email{robert.eymard@u-pem.fr}	
\author[A. Prignet]{Alain Prignet$^3$}
	\email{alain.prignet@u-pem.fr}	
\author[K.S. Talbot]{Kyle S. Talbot$^4$}
	\email{kyle.talbot@gmail.com}
	
\date{\today}

\thanks{The first author was supported by the Australian Research Council's Discovery Projects funding scheme
(project number DP170100605). The fourth author was supported by an Endeavour Research Fellowship from the Australian Government.}


% ------------------------- BEGINNING OF DOCUMENT -------------------------
\maketitle

\begin{abstract}
This article performs a unified convergence analysis of a variety of numerical methods
for a model of the miscible displacement of one incompressible fluid by another
through a porous medium. The unified analysis is enabled through the framework of
the gradient discretisation method for diffusion operators on generic grids. 
We use it to establish a novel convergence result in $L^\infty(0,T; L^2(\Omega))$
of the approximate concentration using minimal regularity assumptions on the solution
to the continuous problem. 
The convection term in the concentration equation is discretised
using a centred scheme. {We present a variety of numerical tests from the literature,
as well as a novel analytical test case. The performance of two schemes are compared on these
tests; both are poor in the case of variable viscosity, small diffusion and medium to small time steps. We show that upstreaming is not a good option to recover stable and accurate solutions, and we propose a correction to recover stable and accurate schemes for all time steps and all ranges of diffusion.}{}
\end{abstract}

% --- INTRODUCTION ---
\section{Introduction} \label{sec:intro}

\subsection{The miscible displacement model} \label{ssec:model}

The single-phase, miscible displacement of one incompressible fluid by another
through a porous medium is described by a nonlinearly-coupled initial-boundary value
problem \cite{cj86,pr62,rw83}. Denote the porous medium by $\O$ and write $(0,T)$ for the time period over which the displacement
occurs. Neglecting density variations of the mixture within the domain, the unknowns are the hydraulic head $p$ (simply called ``pressure'' in this paper) and the concentration $c$ of one of the components
in the mixture, from which one computes the Darcy velocity $\U$ of the fluid mixture. 
The model reads
\begin{subequations} \label{peacemanmodel}
\begin{equation}
	\left.\begin{aligned}
		&\U(x,t) = -\frac{\K(x)}{\mu(c(x,t))}\nabla p(x,t) \\
		&\dive\U(x,t) = (q^I - q^P)(x,t)
	\end{aligned}\right\}, \quad(x,t)\in\O\times(0,T);\label{eq:darcy}
\end{equation} 
\begin{equation} \label{eq:transport}
\Phi(x)\partial_{t}c(x,t)-\dive\left(\D{x}{\U(x,t)}\nabla c - c\U\right)(x,t) = (\hat{c}q^{I} - cq^{P})(x,t),
\quad(x,t)\in\O\times(0,T).
\end{equation}
The reservoir-dependent quantities of porosity and absolute permeability are $\Phi$ 
and $\K$, respectively. The coefficient $\D{x}{\U}$ is the diffusion-dispersion tensor,
and the coefficient $\hat{c}$ is the injected concentration. The sums of injection 
well source terms and production well sink terms are $q^I$ and $q^P$, respectively.
Assuming that the reservoir boundary is impermeable gives the no-flow boundary conditions
\begin{align}
\U(x,t)\cdot\normvec(x) &= 0, \quad(x,t)\in\partial\O\times(0,T),\mbox{ and} \label{eq:noflow1}\\
\D{x}{\U}\nabla c(x,t)\cdot\normvec(x) &= 0, \quad(x,t)\in\partial\O\times(0,T), \label{eq:noflow2}
\end{align}
where $\normvec(x)$ denotes the exterior unit normal to $\partial\O$ at $x\in\partial\O$.
The first of these enforces a compatibility condition upon the source terms:
\begin{equation} \label{eq:compat}
\int_\O q^I(x,t)\ud x = \int_\O q^P(x, t)\ud x\quad \mbox{for all $t\in(0, T)$.} 
\end{equation}
One prescribes the initial concentration
\begin{equation}
c(x,0)=c_{0}(x), \quad x\in\O, \label{eq:initialconc}
\end{equation}
and the system is completed with a normalisation condition on the pressure to eliminate
arbitrary constants in the solution $p$ to \eqref{eq:darcy}:
\begin{equation}
\int_\O p(x,t)\ud x = 0 \quad\mbox{for all $t\in(0, T)$.} \label{eq:normp}
\end{equation}
\end{subequations}
Russell and Wheeler \cite{rw83} give a complete derivation of the system \eqref{peacemanmodel},
which is hereafter called the miscible displacement model.

\subsection{Literature review}
% main novelty of article
The novelty of this article's main result --- Theorem \ref{th:main} --- is twofold. 
It presents what we believe is the first unified 
convergence analysis of a number of numerical schemes for the 
approximation of solutions to \eqref{peacemanmodel}, and it provides a uniform-in-time
strong-in-space convergence property for the concentration.
The unified analysis is performed using a generic framework which, given a classical numerical method (a list of which is given below) for linear elliptic equations, provides a way to design from it a numerical scheme for the miscible displacement problem that ensures convergence.
The uniform-in-time convergence analysis uses a recently-discovered
technique \cite{DE15} developed for scalar degenerate parabolic equations to establish convergence in
$\Leb{\infty}{2}$ (i.e. uniform-in-time) of the approximate concentrations, thereby 
improving upon previous results that establish convergence in $\Leb{2}{2}$ \cite{bjm09,rivwalk11} or in 
$\Leb{q}{r}$ for all $q<\infty$ and $r<2$ \cite{cd07}.

% literature review
The miscible displacement model has a large and diverse numerical literature \cite{fe02}.
Early work by Peaceman \cite{pe66,pe77} and Douglas \cite{douglas70,douglas83} 
uses finite differences. Finite element (FE) and mixed finite element (MFE) methods
for the miscible displacement problem were the subject of considerable
interest in the $1980$s, with several studies conducted by Douglas, Ewing, Russell,
Wheeler and their colleagues \cite{dew83,ew83,ew80,ew84,rw83}. %\cite{darlow84,dew83,douglas80,ew83,ew80,ew84,rw83}. 
%The use of mixed methods stems from the fact that it is $\U$ rather than $p$ that appears explicitly
%in \eqref{eq:transport}, so that a direct approximation of $\U$ (rather than one
%obtained by numerically differentiation of $p$ and then multiplication by the rough
%coefficient $\frac{\K}{\mu}$) is preferred for an accurate approximation of $c$.

In general, different methods for \eqref{eq:darcy} and \eqref{eq:transport} are combined
to produce a scheme for \eqref{peacemanmodel}. Indeed, the convection-dominated nature
of \eqref{eq:transport} leads Russell and others to develop characteristic tracking methods
for handling this equation \cite{erw84,russell82,russell85}. %\cite{erw84,russellphd80,russell82,russell85}.
Related to these are the so-called Eulerian-Lagrangian localised adjoint methods (ELLAMs)
\cite{crhe90,wan00}.
Finite volume (FV) and mixed finite volume (MFV) methods have been studied
for the transport equation alone \cite{ao08} (with a MFE method for the pressure equation)
and the whole system \cite{cd07}, and also Discrete Duality Finite Volume (DDFV) methods
\cite{ckm13,ckm15}. Discontinuous Galerkin (dG) methods are also often employed in the
numerical study of \eqref{peacemanmodel} \cite{bjm09,lirivwalk15,rivwalk11,srw02,wd80}.
%\cite{bjm09,ddwk79,lirivwalk15,rivwalk11,srw02,wd80}.

\subsection{Motivation and framework for the analysis}
% motivation for unified convergence analysis
Given the diversity of methods applied to \eqref{peacemanmodel} and their corresponding
convergences, a natural question to ask is whether we can unify these analyses so
that a single convergence proof holds for all (or at least some) of the methods. 
A unified convergence analysis of this nature requires an appropriate framework; 
one that is sufficiently abstract so as to encompass as many numerical methods as 
possible, but sufficiently concrete to recover existing results for the methods in 
question. Such a framework is the Gradient Discretisation Method (GDM), introduced 
and developed by Droniou, Eymard, Gallou\"et, Guichard, Herbin and their collaborators, 
and which is the subject of a forthcoming monograph \cite{gdmbook} to which we refer frequently.

Section \ref{sec:discprob} gives a reasonably self-contained presentation of the elements of the 
GDM required for the subsequent analysis of \eqref{peacemanmodel}. Following the 
GDM literature, it identifies the four key properties of coercivity, consistency,
limit-conformity and compactness that a numerical method must satisfy
in order for the subsequent proof of Theorem \ref{th:main} in Section \ref{sec:convergence} to apply.
If a numerical method can be written in such a manner that it satisfies these four
properties, then Theorem \ref{th:main} shows that it will approximate
solutions to \eqref{peacemanmodel} with convergences prescribed in the statement of
the theorem. In particular, the approximate concentrations will converge in $\Leb{\infty}{2}$.

% methods that are gradient discretisations; possibly replace list with GDM book reference?
At the time of writing, methods known to satisfy the GDM framework include 
FE with mass lumping \cite{ciarlet91};
the Crouzeix-Raviart non-conforming FE, with or without mass-lumping \cite{cr73,ernguer04};
the Raviart--Thomas MFE \cite{brezfort91}; the Discontinuous Galerkin scheme in its Symmetric Interior Penalty version \cite{dip2012math};
the Multi-Point Flux Approximation (MPFA) O-method \cite{abbm96,edrog98}; 
DDFV methods in dimension two \cite{abh07ddfv,herm03ddfv} and
the CeVeFE--DDFV scheme in dimension three \cite{ch11ddfv};
the Hybrid Mimetic Mixed (HMM) family \cite{degh10unified}, which includes 
the SUSHI scheme \cite{egh10}, Mixed Finite Volumes \cite{de06mixed} and mixed-hybrid
Mimetic Finite Differences (MFD) \cite{bls05mimetic}; nodal MFD \cite{bbl09nodal};
and the Vertex Approximate Gradient (VAG) scheme \cite{egh12}.
Theorem \ref{th:main} therefore applies to these methods, when a centred
discretisation is employed for the convection term.

% technique: compactness
For most nonlinear models, convergence proofs using the GDM framework are based on compactness techniques. 
In contrast to establishing error estimates on the solution --- the method favoured
by most studies in the literature cited above --- such analyses do not require 
uniqueness or regularity of the solution to the continuous problem, assumptions
that are inconsistent with what the physical problem suggests and what the theory
provides (see the discussion in Droniou, Eymard and Herbin \cite{deh16generic}). 
The cost of removing these uniqueness/regularity assumptions is the ability to establish
rates of convergence with respect to discretisation parameters. 
Examples of studies that employ compactness techniques include Chainais-Hillairet--Droniou (MFV)
\cite{cd07}; Amaziane and Ossmani (MFE/FV) \cite{ao08}; Bartels, Jensen and M\"uller (dG) \cite{bjm09}; 
Rivi\`ere and Walkington (dG/MFE) \cite{rivwalk11} and subsequently 
Li, Rivi\`ere and Walkington \cite{lirivwalk15} (dG).

Compactness techniques first establish a priori energy estimates
on the solution to the numerical scheme, which yield weak compactness in the appropriate spaces.  
Ensuring convergence of the numerical solutions for
nonlinear problems such as \eqref{eq:transport} typically requires stronger compactness
than what the estimates alone afford. For time-dependent problems, one obtains such compactness 
by estimating the temporal variation of the numerical solution. From here one may apply
discrete analogues of the Aubin--Simon lemma.

This is the procedure we employ herein. Our estimates on the discrete pressure
and concentration are a straightforward
adaptation of Chainais-Hillairet--Droniou \cite{cd07}. The discrete time derivative
estimate Lemma \ref{lem:dtc} is in the spirit of Droniou--Eymard--Gallou\"et--Herbin \cite{degh13},
and for the convergence of the scheme we adapt many arguments from 
Eymard--Gallou\"et--Guichard--Masson \cite{tp}, who study a related two-phase flow
problem. 

\subsection{Centred discretisation of the convection term} \label{ssec:convection}
% why centred discretisation
The GDM framework offers a generic discretisation of diffusion operators,
but \eqref{peacemanmodel} features a (dominant) convection term $\dive(c\U)$. This
term is usually discretised with an upstream weighted scheme, or occasionally
the aforementioned modified method of characteristics. The motivation for our use of a centred
discretisation is that it is an opportunity to compare the upstream
and centred approaches by the results of numerical experiments conducted using two
simple schemes that we present in Section \ref{ssec:examples}. {The numerical
results provided in Section \ref{sec:experiments} include test cases from the literature, as well
as a novel analytical test case. As expected, they show that both schemes are well-behaved in the presence of sufficient diffusion. With small dispersion and the absence of molecular diffusion, centred schemes also behave rather well for time steps of the same magnitude as those commonly used in the literature. However, they display large instabilities for smaller time steps, and upstream versions similarly do not produce acceptable results. We propose a way to introduce in centred schemes some additional diffusion to recover accurate and stable results for a wide range of time steps. This diffusion is isotropic, scaled by the magnitude of the Darcy velocity, and vanishes with the mesh size.}

Constructing a centred scheme is straightforward in the GDM framework. 
We mention however that including other kinds of advection discretisations
in a unified analysis is possible, by following the ideas of Beir{\~a}o da Veiga, Droniou and Manzini 
\cite{bdvdm11} for HMM methods (or other face-based methods), and Eymard, Feron
and Guichard in the context of incompressible Navier--Stokes \cite{efg16ns}.

% shortcomings of analysis: bounded diffusion-dispersion tensor
Let us finally remark that, to enable the convergence analysis, we apply a truncation (onto the unit interval) to the concentration in the convection term, and we add a boundedness hypothesis on $\DD$ to remove the truncation
from the limit equation (see Remark \ref{rem:DDbounded}).

%The other part of the proof where $c$ must be an admissible test function in the weak
%form of \eqref{eq:transport} is the proof of the $\Leb{\infty}{2}$ convergence
%of the approximate concentrations. Our technique is an adaptation of that introduced
%by Droniou and Eymard \cite{DE15} for the doubly nonlinear parabolic problem that 
%forms the subject of Chapter \ref{ch:jde} of this thesis. In the context of the
%miscible displacement model, the technique relies on the solution $c$ satisfying
%an \emph{energy identity} \eqref{GSeq:energyid} obtained once again by using $c$ as a test function
%in \eqref{eq:transport} and making some straightforward transformations. It seems
%that if $c$ cannot be employed in this role, the technique may not succeed. It
%therefore remains to be seen if the $\Leb{\infty}{2}$ convergence can be extended
%to when $\DD$ takes the form \eqref{eq:ddt}.

\subsection{Notation}
For a topological vector space $X(\O)$ of functions on $\O$, we write $(X(\O))'$ 
for its topological dual. When writing the duality pairing 
$\langle\cdot,\cdot\rangle_{(X(\O))',X(\O)}$, we omit the subscripts if they are clear
from the context. 

When $z\in(1,\infty)$ is a Lebesgue/Sobolev exponent, we write $z'=\frac{z}{z-1}$ for its conjugate. 
We denote by $H_\star^1(\O)$ those elements of $H^1(\O)$ whose integral over $\O$ 
vanishes, equipped with the usual norms. 

We use $C$ to denote an arbitrary positive constant that may change from line to
line. When $C$ appears in an estimate we track only its relevant dependencies 
(or non-dependencies, as is frequently the case).
These dependencies are understood to be nondecreasing. 

Given $x\in\O$ and $\zeta\in\RR^d$, we write $\Dsq{x}{\zeta}$ for the 
square root of $\D{x}{\zeta}$, which is well defined since we always assume that
$\D{x}{\zeta}$ is a symmetric, positive-definite matrix; similarly for $\A$ and
$\A^{1/2}$ defined in the next section.

% --- PRELIMINARIES ---
\section{Problem reformulation and assumptions} \label{sec:prelims}

In order to simplify the presentation, we write the miscible displacement 
model in the following synthesised form, henceforth using $(\pbar,\Ubar,\cbar)$ 
to denote exact solutions and $(p,\U,c)$ to denote approximate 
solutions obtained by the numerical scheme:
% --- REFORMULATION OF PROBLEM ---
\begin{subequations} \label{eq:model}
\begin{align}
&\left.
	\begin{aligned}
	&\dive(\Ubar) = q^I - q^P \quad\mbox{in $\O\times(0,T)$,}&\qquad
	&\Ubar = -\A(\cdot,\cbar)\nabla\pbar \quad\mbox{in $\O\times(0,T)$,}\\
	&\int_\O\pbar(x,\cdot)\ud x = 0 \quad\mbox{on $(0,T)$,}&\qquad
	&\Ubar\cdot\normvec=0 \quad\mbox{on $\partial\O\times(0,T)$,}
	\end{aligned}
\right\} \label{eq:elliptic}\\
&\left.
	\begin{aligned}
	&\Phi\partial_{t}\cbar - \dive(\D{\cdot}{\Ubar}\nabla\cbar) +\dive(\cbar\,\Ubar)
	= \hat{c}q^I - \cbar q^{P}\quad\mbox{in $\O\times(0,T)$,}\\
	&\cbar(\cdot,0) = c_0 \quad\mbox{in $\O$,}\\
	&\D{\cdot}{\Ubar}\nabla\cbar\cdot\normvec = 0 \quad\mbox{on $\partial\O\times(0,T)$.}
	\end{aligned}
\right\} \label{eq:parabolic}
\end{align}
\end{subequations}
% --- HYPOTHESES ON DATA ---
\begin{subequations}\label{assumptions}
Our assumptions on the data are then as follows.
% domain
\bhyp{hyp:domain}
\O \mbox{ is a bounded, connected {polytopal} subset of $\RR^d$, $d=1,2$ or $3$, and
$T>0$.}
\ehyp
Denote by $\cS_d(\RR)$ the set of $d\times d$ symmetric matrices with real entries. 
The tensor $\A$ encodes the absolute permeability $\K$ and viscosity $\mu$:
% combined permeability/viscosity tensor
\bhyp{hyp:A}
\A:\O\times\RR\to \cS_d(\RR)\mbox{ is a Carath\'eodory function such that}\\
\exists\alpha_{\A}>0, \, \exists \Lambda_{\A}>0 
\mbox{ such that, for a.e. $x\in\O$, all $s\in\RR$ and all $\xi\in\RR^d$,}\\
\A(x,s)\xi\cdot\xi\geq\alpha_{\A}|\xi|^2 \mbox{ and }|\A(x,s)|\leq\Lambda_{\A}. 
\ehyp
% diffusion-dispersion tensor
We assume that the diffusion-dispersion tensor satisfies
\bhyp{hyp:D}
\DD:\O\times\RR^{d}\to\cS_d(\RR) \mbox{ is a Carath\'eodory function such that}\\
\exists \alpha_{\DD} > 0, \, \exists \Lambda_{\DD} > 0 \mbox{ such that, for a.e. $x\in\O$ and all }\zeta, \xi \in \RR^{d},\\
\D{x}{\zeta}\xi\cdot\xi\geq\alpha_{\DD}|\xi|^{2}\mbox{ and }
|\D{x}{\zeta}|\leq \Lambda_{\DD}.
\ehyp
The assumptions on the porosity, injected concentration and initial concentration are
standard:
% porosity
\bhyp{hyp:porosity}
\Phi \in L^{\infty}(\Omega) \mbox{ and there exists } \phi_{\ast} > 0 \mbox{ such that for a.e. $x \in \Omega$, }
\phi_{\ast} \leq \Phi(x) \leq \phi_{\ast}^{-1};
\ehyp
% injected concentration
\bhyp{hyp:injectedconc}
\hat{c}\in L^{\infty}(\O\times(0,T))\mbox{ satisfies }0\leq\hat{c}(x, t)\leq 1 \mbox{ for a.e. }(x, t)\in\O\times(0, T);
\ehyp
% initial concentration
\bhyp{hyp:initialconc}
c_{0}\in L^{\infty}(\O)\mbox{ satisfies }0\leq c_{0}(x)\leq 1\mbox{ for a.e. }x\in\O.
\ehyp
% source terms
Finally, we assume that
\bhyp{hyp:source}
q^I, q^P \in L^{\infty}(\Omega\times(0,T)) \mbox{ are nonnegative and such that }\\
\int_\O q^I(x,\cdot)\ud x= \int_\O q^P(x,\cdot)\ud x \mbox{ a.e. in }(0,T).
\ehyp
\end{subequations}
% remark: diffusion-dispersion tensor
\begin{remark}\label{rem:ddt}
Peaceman showed \cite{pe66} that the diffusion-dispersion tensor takes the form
\begin{equation} \label{eq:ddt}
\D{x}{\U} = \Phi(x)\bigg(\dm\IdM + |\U|\Big(\dl \E(\U) + \dt(\IdM - \E(\U))\Big)\bigg),
\quad\mbox{with}\quad
\E(\U)= \left(\frac{\U_i\U_j}{|\U|^2}\right)_{1\leq i,j\leq d}.
\end{equation}
Here $\IdM$ is the $d$-dimensional identity matrix, $\dm>0$ is the molecular diffusion coefficient,
and $\dl>0$ and $\dt>0$ are the longitudinal and transverse mechanical dispersion
coefficients, respectively. 

Although \eqref{eq:ddt} satisfies the coercivity condition in {\eqref{hyp:D}}, it is
not uniformly bounded. Indeed, one can show that
$|\D{x}{\U}|\leq C(1+|\U|)$, where $C>0$. The necessity of our stronger assumption
on $\DD$ can be traced to our choice of discretisation for the convection term in 
\eqref{eq:parabolic}; we discuss this further below.
\end{remark}
We henceforth supress the dependence on space of $\A$ and $\DD$ from the notation by writing
$\A(x,c)=\A(c)$ and $\D{x}{\U}=\DD(\U)$.
%remark: bounded sources
\begin{remark}\label{rem:sources}
The assumption that the sources are bounded is primarily to simplify the analysis.
Indeed, it suffices to take $q^I\in\Leb{\infty}{2}$ and $q^P\in\Leb{\infty}{r}$, where $r>1$ if
$d=2$ and $r\geq\frac{3}{2}$ if $d=3$.
These spatial regularities on $q^P$ arise from the need to bound the production well term in the discrete
time derivative estimate \eqref{est:dtc} below, which can be accomplished using a discrete
Sobolev inequality \cite[Lemma B.24]{gdmbook}. Employing this inequality, one can 
improve the spatial regularity of both the discrete test function and the 
approximation to $\cbar$ from $L^2$ to $L^q$ for all $q<\infty$ (if $d=2$) or to
$L^6$ (if $d=3$). Whilst most of the schemes that we consider in the Gradient Discretisation 
Method framework satisfy such a discrete Sobolev inequality,
its sole usage herein would be in the estimate \eqref{est:dtc}.
\end{remark}

The main result of this article demonstrates that our approximate solutions converge
to the following notion of weak solution to \eqref{eq:model}, the existence of which
is due to Feng \cite{fe95}, Chen and Ewing \cite{ce99} (both of whom assume Peaceman's 
diffusion-dispersion tensor) and Fabrie and Gallou\"et \cite{fg00} (who use the assumption
\eqref{hyp:D}).
% --- WEAK SOLUTION ----
\begin{definition} \label{def:weaksol}
Assume \eqref{assumptions}. A weak solution to \eqref{eq:model}
is a pair of functions $(\pbar,\cbar)$ satisfying
\begin{equation} \label{eq:weaksol}
	\left.
	\begin{gathered}
	\cbar\in L^2(0,T; H^1(\O))\cap C([0,T];L^2(\O)),\quad 0\leq\cbar(x,t)\leq1\quad\mbox{for a.e. $(x,t)\in\O\times(0,T)$;}\\
	\Phi\partial_{t}\cbar\in L^2(0,T; (H^{1}(\O))'),\quad \cbar(\cdot,0)=c_0\quad\mbox{in $L^2(\O)$};\\
	\pbar\in L^\infty(0,T;H^1_{\star}(\O)),\quad\Ubar(x,t)=-\A(\cbar(x,t))\nabla\pbar(x,t);\\
	-\int_0^T\int_\O\Ubar(x,t)\cdot\nabla\varphi(x,t)\ud x\ud t\\
	= \int_0^T\int_\O \left( q^I - q^P\right)(x,t)\varphi(x,t)\ud x\ud t 
	\quad\forall\varphi\in L^{1}(0,T; H^1(\O));\\
	\int_0^T\langle\Phi\partial_{t}\cbar(\cdot,t),\psi(\cdot,t)\rangle\ud t
	+ \int_0^T\int_\O\DD(\Ubar(x,t))\nabla\cbar(x,t)\cdot\nabla\psi(x,t)\ud x\ud t\\
	- \int_0^T\int_\O\cbar(x,t)\,\Ubar(x,t)\cdot\nabla\psi(x,t)\ud x\ud t \\
	= \int_0^T\int_\O\left( \hat{c}q^I - \cbar q^P\right)(x,t)\psi(x,t)\ud x\ud t,
	\quad\forall\psi\in L^{2}(0,T; H^{1}(\O)).
	\end{gathered} 
	\right\}
\end{equation}
\end{definition}
There are two noteworthy features to this definition. First, the regularity of $\cbar$ matches
that of $\varphi$, so one can take $\varphi=\cbar$. 
In doing so, by integrating by parts on the time derivative term and using the elliptic 
equation to transform the convective term (see the proof of \cite[Proposition 3.1]{dt14}), 
it is straightforward to show that for every $T_0\in(0,T)$, the solution
satisfies the identity
\begin{equation} \label{eq:energyid}
	\begin{gathered} 
	\frac{1}{2}\int_\O\Phi(x)\left(\cbar(x,T_0)^2 - c_0(x)^2\right)\ud x
	= \int_0^{T_0}\int_\O\left(\cbar\hat{c}q^I\right)(x,t)\ud x\ud t \\
	- \frac{1}{2}\int_0^{T_0}\int_\O \left(\cbar^2(q^I + q^P)\right)(x,t)\ud x\ud t
	- \int_0^{T_0}\int_\O\DD(\Ubar(x,t))\nabla\cbar(x,t)\cdot\nabla\cbar(x,t)\ud x\ud t.
	\end{gathered}
\end{equation}
We will see that the ability to take $\cbar$ as a test function in the transport
equation is critical to ensuring our generic discretisations properly approximate
\eqref{eq:weaksol}, and this resultant identity enables improved temporal convergence
of the approximation to $\cbar$.

We consider a solution to be a pair
$(\pbar,\cbar)$ rather than a triple $(\pbar,\Ubar,\cbar)$. This is inconsequential,
and comes from the fact that our generic discretisation below requires only approximations 
of $\pbar$ and $\cbar$, from which we obtain an approximation $\Ubar$ by the discrete
gradient operator. Note however that many finite element and finite volume schemes
for \eqref{eq:model} approximate $\Ubar$ directly, since numerical differentiation
of the approximation to $\pbar$ can lead to rather poor approximations of the Darcy velocity \cite{darlow84,erw84}.

% --- GRADIENT DISCRETISATION ---
\section{Discrete problem and main result} \label{sec:discprob}

The basic objects of study in our GDM discretisation of the miscible displacement 
problem are the {gradient discretisation}, which must satisfy the four properties of {coercivity}, {consistency},
{limit-conformity} and {compactness} in order to guarantee convergence
of the associated {gradient scheme}.

% --- GRADIENT DISCRETISATION ---
\subsection{Gradient discretisation} \label{ssec:graddisc}

\begin{definition}[Gradient discretisation] \label{def:gradient_disc}
A {gradient discretisation} for the miscible displacement problem is a family
$\disc=(\unknowns,\PiD,\gradD,\interpD,(t^{(n)})_{n=0,\ldots,N})$, where
\begin{enumerate}[(i)]
\item The set $\unknowns$ of discrete unknowns is a finite-dimensional vector space
over $\RR$.
\item $\PiD:\unknowns\to L^2(\O)$ is a linear mapping, called the {function reconstruction operator}.
\item $\gradD:\unknowns\to L^2(\O)^d$ is a linear mapping called the
{gradient reconstruction operator},
and must be chosen so that $\norm{\,\cdot\,}_{\disc,\ellip}$ defined by
\[
\norm{w}_{\disc,\ellip} =
\left[ \norm{\gradD w}^2_{L^2(\O)^d} + \left(\int_\O \PiD w(x)\ud
x\right)^2\right]^{1/2}
\]
is a norm on $\unknowns$. We also define
\[
\norm{w}_{\disc,\para} =
\left[\norm{\PiD w}^2_{L^2(\O)} + \norm{\gradD w}^2_{L^2(\O)^d}\right]^{1/2}.
\]
\item $\interpD:L^2(\O)\to\unknowns$ is a linear interpolation operator.
\item $0=t^{(0)}<t^{(1)}<\ldots<t^{(N)}=T$. 
\end{enumerate}
\end{definition}
Note that besides the use of the $\norm{\,\cdot\,}_{\disc,\para}$ norm in (iii), this notion of gradient discretisation is
identical to the space-time gradient discretisation for parabolic Neumann problems
presented in \cite[Definitions 3.1 and 4.1]{gdmbook}.
  
The subscripts `$\ellip$' and `$\para$' denote {elliptic} and {parabolic}, respectively,
and reflect the fact that we use $\norm{\,\cdot\,}_{\disc,\ellip}$ to estimate the discrete pressure
and $\norm{\,\cdot\,}_{\disc,\para}$ to estimate the discrete concentration. 
The manner in which we incorporate the zero-mean value
of the pressure into the scheme (see \eqref{scheme_p}) necessitates the use of the 
$\norm{\,\cdot\,}_{\disc,\ellip}$ norm and its associated Poincar\'e inequality
\eqref{eq:poincare} to obtain a spatial $L^2$ estimate on the discrete pressure. 
The presence of the time derivative in the parabolic equation affords the same estimate
without the use of a Poincar\'e inequality, hence the $\norm{\,\cdot\,}_{\disc,\para}$ norm.

Next we introduce some notation. Consider $n\in\{0,\ldots,N-1\}$, 
$t\in(t^{(n)},t^{(n+1)}]$ and $w=(w^{(n)})_{n=0,\ldots,N-1}\in \unknowns^N$. 
Set $\deltat=t^{(n+1)}-t^{(n)}$, $\deltatD=\max_{n=0,\ldots,N-1}\deltat$ and define
the {discrete time derivative}
\begin{equation*}
\deltaD w(t):=
\deltaD^{(n+\frac{1}{2})}w
:= \PiD\frac{w^{(n+1)}-w^{(n)}}{\deltat}\in L^2(\Omega).
\end{equation*}
We use the notation $\PiDe$, $\PiD$, $\gradD$  (with `exp' for explicit) for functions 
dependent on both space and time as follows. For almost-every $x\in\O$, for all 
$n\in\{0,\ldots,N-1\}$ and all $t\in(t^{(n)}, t^{(n+1)}]$, we set 
\begin{equation}\label{imp.ex}
	\begin{gathered}
	\PiD w(x,0)=\PiDe w(x,0)=\PiD w^{(0)}(x), \quad 
	\PiDe w(x,t)=\PiD w^{(n)}(x),\\
	\PiD w(x,t)=\PiD w^{(n+1)}(x), \mbox{ and }
	\gradD w(x,t)=\gradD w^{(n+1)}(x).
	\end{gathered}
\end{equation}
Normally only one choice of evaluation is required, either implicit ($\PiD w(x,t)=\PiD w^{(n+1)}(x)$)
or explicit ($\PiD w(x,t)=\PiD w^{(n)}(x))$. Again, the coupled nature of the miscible
displacement problem appears to necessitate the use of both. We discuss this further below.

Before introducing the scheme, one final remark is necessary. The solution $\cbar$
to \eqref{eq:parabolic} satisfies a maximum principle: $0\leq\cbar\leq 1$. We cannot in general
prove such an {a priori} estimate on the numerical approximation of $\cbar$ in the GDM
framework, except in very specific cases such as the two-point finite
volume scheme on simple meshes and with a diffusion-dispersion tensor that does
not depend on $\U$ \cite{review,tp}. 
Indeed, in Section \ref{sec:experiments} we present numerical results on coarse meshes 
that exhibit values of the concentration outside the unit interval. 
To establish the basic discrete energy estimates of Lemma \ref{lem:parabolicest} 
on the numerical solution, it is therefore necessary for us to stabilise the scheme 
by means of a truncation operator applied to $\cbar$ in the convection term. 
To this end, for $s\in\RR$ define the truncation onto the unit interval by
\begin{equation} \label{eq:trunc}
\trunc(s) := \max(0, \min(s,1)).
\end{equation} 
The scheme for \eqref{eq:model} is then obtained by replacing the continuous spaces
and operators in \eqref{eq:weaksol} with their discrete analogues.

% --- SCHEME FOR PEACEMAN MODEL ---
\begin{definition}[Gradient scheme for \eqref{eq:model}] \label{def:scheme}
Find sequences $p=\left(p^{(n)}\right)_{n=1,\ldots,N}\in\unknowns^N$ and
$c=\left(c^{(n)}\right)_{n=0,\ldots,N}\in\unknowns^{N+1}$ such that
$c^{(0)}=\cI_\disc c_0$ and for all $n=0,\ldots,N-1$,
\begin{subequations} \label{scheme}
\begin{equation} \label{scheme_p}
	\left.
	\begin{gathered}
	\Udisc^{(n+1)}(x) = -\A\left(\PiD c^{(n)}(x)\right)\gradD p^{(n+1)}(x),\\
	-\int_\O\Udisc^{(n+1)}(x)\cdot\gradD w(x) \ud x 
	+ \left( \int_\O \PiD p^{(n+1)}(x)\ud x\right)\left( \int_\O \PiD w(x)\ud x\right)\\
	= \frac{1}{\deltat}\int_{t^{(n)}}^{t^{(n+1)}}\int_\O (q^I - q^P)(x,t)\PiD w(x)\ud x\ud t\quad\forall w\in\unknowns,
	\end{gathered} \right\}
\end{equation}
\begin{equation} \label{scheme_c}
	\begin{gathered}
	\int_\O\left(\Phi(x)\deltaD^{(n+\frac{1}{2})}c(x)\PiD w(x) 
	+ \DD(\Udisc^{(n+1)}(x))\gradD c^{(n+1)}(x)\cdot\gradD w(x)\right)\ud x\\
	- \int_\O \trunc\left(\PiD c^{(n+1)}(x)\right)\Udisc^{(n+1)}(x)\cdot\gradD w(x)\ud x \\
	=  \frac{1}{\deltat}\int_{t^{(n)}}^{t^{(n+1)}}\int_\O
	\left(\hat{c}(x)q^I(x,t) - \PiD c^{(n+1)}(x)q^P(x,t)\right)\PiD w(x)\ud x\ud t\quad\forall w\in\unknowns.
	\end{gathered} 
\end{equation}
\end{subequations}
\end{definition}

Note the choice of $\PiD c^{(n)}$ (i.e. explicit in time) in the definition of the
discrete Darcy velocity. This choice follows Chainais-Hillairet--Droniou \cite{cd07}
and decouples the scheme for \eqref{eq:elliptic} from the scheme for \eqref{eq:parabolic}.
This facilitates the proof of existence of solutions to \eqref{scheme}, but is by no means
necessary. It does, however, reflect a structural choice common to many schemes in
the literature on the miscible displacement problem.
The second integral term in \eqref{scheme_p} accounts for the zero mean value of $\pbar$.

In order to prove the main result of this article, our gradient discretisations must
satisfy properties that mimic as much as possible the properties of the continuous operators. 
The first of these, {coercivity}, imposes a restriction on the $L^2$ interaction
between $\PiD$ and $\gradD$. In particular, it gives us a {discrete Poincar\'e 
inequality} and ensures stability of the underlying method.
%coercivity
\begin{definition}[Coercivity] \label{def:coercivity}
Let 
\begin{equation*}
C_\disc =  \max_{w\in\unknowns\setminus\{0\}}\frac{\norm{\PiD
w}_{L^2(\O)}}{\norm{w}_{\disc,\ellip}}.
\end{equation*}
A sequence $(\disc_m)_{m\in\NN}$ of gradient discretisations is {coercive} if
there exists $C_P\in\RR^+$ such that for all $m\in\NN$, $C_{\disc_m}\leq C_P$.
\end{definition}
The corresponding Poincar\'e inequality for the $\norm{\,\cdot\,}_{\disc,\ellip}$ norm 
is then
% remark: discrete Poincare inequality
\begin{equation} \label{eq:poincare}
\norm{\PiD w}_{L^2(\O)} 
\leq C_\disc\norm{w}_{\disc,\ellip}. 
\end{equation}
The next property, consistency, ensures that we can recover our (spatial)
solution space $H^1(\O)$ to arbitrary $L^2$ precision using reconstructed functions
and their gradients from $\unknowns$. In this sense, it shows that $\unknowns$ is a
`good sample' of $H^1(\O)$. This property also ensures the recovery of the initial condition,
and the convergence to 0 of the time steps.
%consistency
\begin{definition}[Consistency] \label{def:consistency}
For $\varphi\in H^1(\O)$, define the map $S_\disc:H^1(\O)\to[0,\infty)$ by
\begin{equation*}
S_\disc(\varphi)=\min_{w\in\unknowns}\left( \norm{\PiD w - \varphi}_{L^2(\O)}
+ \norm{\gradD w - \nabla\varphi}_{L^{2}(\O)^d}\right).
\end{equation*}
A sequence $(\disc_m)_{m\in\NN}$ of gradient discretisations is {consistent} if,
as $m\to\infty$,
\begin{itemize}
\item for all $\varphi\in H^1(\O)$, $S_{\disc_m}(\varphi)\to0$,
\item for all $\varphi\in L^2(\O)$, $\PiDm \interpDm \varphi\to\varphi$ strongly in $L^2(\O)$, and
\item $\deltatDm \to 0$.
\end{itemize}
\end{definition}
Elements of the continuous spaces in our problem satisfy a {divergence} formula.
The quantity $W_\disc$ defined below measures the error introduced into this formula by
the discretisation method. For convergence of the schemes, we require that the formula
is satisfied asymptotically. 
%limit conformity
\begin{definition}[Limit-conformity] \label{def:limconform}
Let $W = \left\{ \bvarphi\in C^\infty(\ol{\O})^d : \bvarphi\cdot\normvec=0\mbox{ on }\partial\O\right\}$.
For $\bvarphi\in W$, define $W_\disc:W\to[0,\infty)$ by
\begin{equation*}
W_\disc(\bvarphi)=\max_{w\in\unknowns\setminus\{0\}}\frac{1}{\norm{w}_{\disc,\ellip}}
\left| \int_\O \left(\gradD w(x)\cdot\bvarphi(x) + \PiD w(x)\dive\bvarphi(x)\ud x\right)\right|.
\end{equation*}
A sequence $(\disc_m)_{m\in\NN}$ of gradient discretisations is {limit-conforming} if
for all $\bvarphi\in W$, $W_{\disc_m}(\bvarphi)\to0$ as $m\to\infty$.
This implies also the same property with
$\norm{w}_{\disc,\para}$ instead of $\norm{w}_{\disc,\ellip}$.
\end{definition}
%remark: necessity of vanishing normal trace
The condition of vanishing normal trace in $W$ is imposed by the homogeneous Neumann boundary
conditions of the miscible displacement model \cite{gdmbook}.

Our final requirement on the gradient discretisations is that the operators $\PiD$
and $\gradD$ afford us a compactness property. 
%compactness
\begin{definition}[Compactness] \label{def:compact}
A sequence $(\disc_m)_{m\in\NN}$ of gradient discretisations is
{compact} if for all sequences $(u_m)_{m\in\NN}$ with
$u_m\in\unknownsm$ such that $(\norm{u_m}_{\disc,\para})_{m\in\NN}$ or
$(\norm{u_m}_{\disc,\ellip})_{m\in\NN}$ is bounded, the sequence
$\left(\PiDm u_m\right)_{m\in\NN}$
has a subsequence converging in $L^2(\O)$.
\end{definition}
With these definitions in place, the main result of this article is the following theorem.

% --- CONVERGENCE THEOREM ---
\begin{theorem}\label{th:main}
Assume \eqref{assumptions} and let $(\discm)_{m\in\NN}$ be a sequence of gradient
discretisations that is coercive, consistent, limit-conforming and compact. For $m\in\NN$,
let $(\pDm, \cDm)$ be a solution to the gradient scheme \eqref{scheme}
with $\disc=\discm$. Then there
exists a weak solution $(\pbar,\cbar)$ of \eqref{eq:model} in the sense of 
Definition \ref{def:weaksol} and a subsequence of gradient discretisations, again denoted
$(\discm)_{m\in\NN}$, such that as $m\to\infty$,
\begin{enumerate}[(i)]
\item $\PiDm\pDm \to \pbar$ strongly in $\Leb{p}{2}$ for every $p<\infty$ and weakly-$\ast$
in $\Leb{\infty}{2}$;
\item $\gradDm\pDm \to \nabla\pbar$ and 
$\UDm:=-\A\left(\PiDme\cDm\right)\gradDm\pDm\to\Ubar$,
both strongly in $\Leb[d]{2}{2}$;
\item $\PiDm\cDm\to\cbar$ and $\PiDme\cDm\to\cbar$ strongly in $\Leb{\infty}{2}$;
\item $\gradDm\cDm\weakto\nabla\cbar$ weakly in $\Leb[d]{2}{2}$.
\end{enumerate}
\end{theorem}

% --- TWO EXAMPLES OF GRADIENT SCHEMES ---
\subsection{Two examples}\label{ssec:examples}
We present here two concrete examples that fit into the GDM framework.

\subsubsection{Scheme A: finite-difference scheme}

Scheme A, defined only on rectangular meshes, leads to a five-point finite volume scheme in the case of isotropic diffusion problems
on rectangular meshes. For the sake of simplicity, we describe this scheme in the case where $\Omega = (0,L)^2$, with $L>0$.
For $N\in\NN$, set $h = L/N$, and for $i,j=1,\ldots,N$, set $K_{ij} = ((i-1)h,ih)\times((j-1)h,jh)$. We then set
\begin{enumerate}[(i)]
\item $X_\disc = \{ (u_{ij})_{i,j=1,\ldots,N}\in \RR^{N^2}\}$;
\item For every $u\in X_\disc$, $i,j=1,\ldots,N$ and for almost-every $x\in K_{ij}$, 
$\PiD u(x) = u_{ij}$ (piecewise constant reconstruction in all $K_{ij}$);
\item For every $u\in X_\disc$, for $i,j=1,\ldots,N$, set $u_{N+1,j} = u_{Nj}$, 
$u_{i,N+1} = u_{iN}$, $u_{0j} = u_{1j}$, and $u_{i0} = u_{i1}$. For $i,j=1,\ldots,N$, we
then define the piecewise constant reconstruction of the gradient by
	\begin{enumerate}
	\item  $\gradD u(x)=(\frac{1}{h}(u_{i+1,j}-u_{ij}), \frac{1}{h}(u_{i,j+1}-u_{ij}))$ on $((i-\frac{1}{2})h,ih)\times((j-\frac{1}{2})h,jh)$,
	\item  $\gradD u(x)=(\frac{1}{h}(u_{i+1,j}-u_{ij}), \frac{1}{h}(u_{i,j}-u_{i,j-1}))$ on $((i-\frac{1}{2})h,ih)\times((j-1)h,(j-\frac{1}{2})h)$,
	\item  $\gradD u(x)=(\frac{1}{h}(u_{ij}-u_{i-1,j}), \frac{1}{h}(u_{i,j+1}-u_{ij}))$ on $((i-1)h,(i-\frac{1}{2})h)\times((j-\frac{1}{2})h,jh)$,
	\item  $\gradD u(x)=(\frac{1}{h}(u_{ij}-u_{i-1,j}), \frac{1}{h}(u_{i,j}-u_{i,j-1}))$ on $((i-1)h,(i-\frac{1}{2})h)\times((j-1)h,(j-\frac{1}{2})h)$.
 	\end{enumerate}
Figure \ref{fig:discgrad} provides a visualisation of this construction. 
\end{enumerate}
For proofs that this scheme is coercive, consistent, limit conforming and compact 
in the case of homogeneous Dirichlet boundary conditions, see Droniou, Eymard and Feron \cite{def15}. 
The adaptation of their arguments to homogeneous Neumann boundary conditions is straightforward.
\begin{figure}[]
\begin{center}
 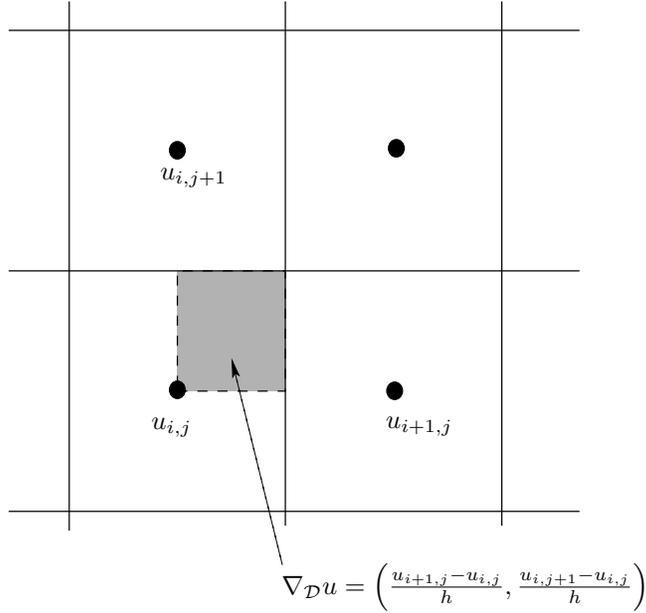
\caption{Discrete gradient for Scheme A}
\label{fig:discgrad}
\end{center}
\end{figure}
 
\subsubsection{Scheme B: mass-lumped $P^1$ conforming scheme}

Scheme B, defined only for conforming triangular meshes, is the standard $P^1$ conforming finite element scheme 
with mass lumping on the dual mesh obtained by joining the center of gravity of the 
triangles with the middle of the edges \cite[Section 8.4]{gdmbook} (dual meshes can also
easily be defined in 3D, so this scheme is actually also defined for 3D simplicial meshes). Denote by $\mathcal{V}$ the set of the 
vertices of the mesh, and for $i\in \mathcal{V}$ define $K_i$ as the dual cell 
around the vertex $i$. We then set
\begin{enumerate}[(i)]
\item $X_\disc = \{ u=(u_{i})_{i\in \mathcal{V}}\in \RR^{\mathcal{V}}\}$;
\item For every $u\in X_\disc$, $i\in \mathcal{V}$ and for almost-every $x\in K_{i}$, 
$\PiD u(x) = u_{i}$ (piecewise constant reconstruction in all $K_{i}$);
\item For every $u\in X_\disc$, define $\gradD u$ as the gradient of the conforming 
piecewise affine function reconstructed in the triangles from the values at the 
vertices of the triangles.
\end{enumerate}
For proofs that this scheme satisfies the four properties above in the case of homogeneous Dirichlet
boundary conditions, see Droniou, Eymard and Herbin \cite{deh16generic}. Their adaptation to 
homogeneous Neumann boundary conditions is once again straightforward (see also
\cite[Sections 8.3 and 8.4]{gdmbook}).

We revisit these schemes in Section \ref{sec:experiments} to present the results
of numerical experiments.

% --- ESTIMATES ---
\section{Estimates} \label{sec:est}
Thanks to the gradient discretisation framework, the following elliptic estimates are 
very similar to their continuous analogues.
To reiterate the conventions outlined in the introduction, for the constants appearing
in estimates we highlight only their relevant dependencies. In the current setting,
this amounts to highlighting scheme-dependent quantities 
and demonstrating that the additional regularity assumed on the sources 
really is primarily for convenience.

% discrete pressure, Darcy velocity estimates
\begin{lemma}\label{lem:ellipticest}
Assume \eqref{assumptions} and let $\disc$ be a gradient
discretisation. Let $(p,c)$ be a solution to the gradient scheme for 
\eqref{eq:model}. Then there exists
$C$ depending only on 
%$\alpha_\A$, $\Lambda_\A$, 
$C_P\geq C_\disc$ and $\norm{q^I - q^P}_{\Leb{\infty}{2}}$ such that
\begin{equation}\label{est:elliptic}	
\norm{\PiD p}_{\Leb{\infty}{2}}\leq C, \quad
\norm{\gradD p}_{\Leb[d]{\infty}{2}}\leq C \quad\mbox{and}\quad
\norm{\Udisc}_{\Leb[d]{\infty}{2}}\leq C.
\end{equation}
As a consequence, for a given $n\in\{0,\ldots,N-1\}$, there exists one and only one 
$p^{(n+1)}\in\unknowns$ such that \eqref{scheme_p} holds.
\end{lemma}
\begin{proof}
Let $n\in\{0,\ldots,N-1\}$ and take $w=p^{(n+1)}$ in \eqref{scheme_p} to obtain
\begin{multline*}
\int_\O \A\left(\PiD c^{(n)}(x)\right)\gradD p^{(n+1)}(x)\cdot\gradD p^{(n+1)}(x)\ud x
+ \left( \int_\O \PiD p^{(n+1)}(x)\ud x\right)^2 \\
= \frac{1}{\deltat}\int_{t^{(n)}}^{t^{(n+1)}}\int_\O \left(q^I - q^P\right)(x,t)\PiD p^{(n+1)}(x)\ud x\ud t.
\end{multline*}
Next, apply the coercivity \eqref{hyp:A} of $\A$ and H\"older's inequality:
\begin{equation*}
\alpha_\A\int_\O \left| \gradD p^{(n+1)}(x)\right|^2\ud x
+ \left(\int_\O \PiD p^{(n+1)}(x)\ud x\right)^2
\leq \norm{q^I - q^P}_{\Leb{\infty}{2}}\norm{\PiD p^{(n+1)}}_{L^{2}(\O)}.
\end{equation*}
Then using the discrete Poincar\'e inequality \eqref{eq:poincare}, we have
\begin{equation} \label{eq:zerokernel}
\int_\O \left| \gradD p^{(n+1)}(x)\right|^2\ud x
+ \left(\int_\O \PiD p^{(n+1)}(x)\ud x\right)^2
\leq \left(\frac{C_\disc}{\min(1,\alpha_\A)}\right)^2\norm{q^I - q^P}_{\Leb{\infty}{2}}^2.
\end{equation}
This yields the estimate on $\gradD p$ and, thanks again to \eqref{eq:poincare},
the estimate on $\PiD p$. The estimate on $\Udisc$ follows from the estimate on $\gradD p$
and the bound on $\A$.

For a given $n\in\{0,\ldots,N-1\}$, \eqref{scheme_p} is a square linear system. 
If the right-hand side of \eqref{scheme_p} vanishes, \eqref{eq:zerokernel} shows 
that the system has a trivial kernel.
The existence and uniqueness of the solution $p^{(n+1)}\in\unknowns$ to \eqref{scheme_p}
then follows immediately.
\end{proof}

The following discrete energy estimates are also a reasonably straightforward translation of the
continous estimates \cite[Proposition 3.1]{dt14}. Note however that unlike the continuous estimates, due to
our use of the truncation operator we cannot follow \cite{fg00} by using the pressure
equation to transform the convection term.
% discrete concentration + its gradient, diffusion-dispersion term estimates
\begin{lemma}\label{lem:parabolicest}
Assume \eqref{assumptions} and let $\disc$ be a gradient
discretisation. Let $(p,c)$ be a solution to the gradient scheme for 
\eqref{eq:model}. Then there exists
$C$ depending only on 
%$T$, $\alpha_\A$, $\Lambda_\A$, $\alpha_\DD$, $\phi_\ast$, 
$C_P\geq C_\disc$, $\norm{q^I}_{\Leb{\infty}{2}}$, $\norm{q^P}_{\Leb{\infty}{2}}$ and 
$C_I\geq\norm{\PiD c^{(0)}}_{L^2(\O)}$ such that
\begin{equation}\label{est:parabolic}	
\norm{\PiD c}_{\Leb{\infty}{2}}\leq C,\quad
\norm{\PiDe c}_{\Leb{\infty}{2}}\leq C\quad\mbox{and}\quad
\norm{\gradD c}_{\Leb[d]{2}{2}}\leq C.
\end{equation}	
\end{lemma}
\begin{proof}
Take $n\in\{0,\ldots, N-1\}$ and $w=c^{(n+1)}$ in \eqref{scheme_c} to obtain
\begin{multline} \label{eq:1stenergy}
\int_\O\left(\Phi(x)\deltaD^{(n+\frac{1}{2})}c(x)\PiD c^{(n+1)}(x) 
+ \DD(\Udisc^{(n+1)}(x))\gradD c^{(n+1)}(x)\cdot\gradD c^{(n+1)}(x)\right)\ud x \\
- \int_\O \trunc\left(\PiD c^{(n+1)}(x)\right)\Udisc^{(n+1)}(x)\cdot\gradD c^{(n+1)}(x)\ud x
+ \frac{1}{\deltat}\int_{t^{(n)}}^{t^{(n+1)}}\int_\O\left(\PiD c^{(n+1)}(x)\right)^2 q^P(x,t)\ud x\ud t\\
= \frac{1}{\deltat}\int_{t^{(n)}}^{t^{(n+1)}}\int_\O \hat{c}(x,t)q^I(x,t)\PiD c^{(n+1)}(x)\ud x\ud t.
\end{multline}
For $a$, $b\in\RR$, {$(a-b)a = \frac{1}{2}(a^2-b^2)+\frac{1}{2}(a-b)^2$}. Applying this to the discrete
time derivative term yields
\begin{equation}
\begin{aligned}
\left(\PiD c^{(n+1)}-\PiD c^{(n)}\right)\PiD c^{(n+1)}
={}& 
\frac{1}{2}\left((\PiD c^{(n+1)})^2 - (\PiD c^{(n)})^2\right) \\
&{+ \frac{1}{2}\left(\PiD c^{(n+1)} - \PiD c^{(n)}\right)^2}.
\end{aligned}
\label{est:parabolic_disct}\end{equation}
Then multiplying \eqref{eq:1stenergy} by $\deltat$ and summing over $n=0,\ldots,m-1$ 
for some $m\in[1,N]$, we have
\begin{multline} \label{eq:2ndenergy}
\frac{1}{2}\int_\O\Phi(x)\left[\left(\PiD c^{(m)}(x)\right)^2 -\left(\PiD c^{(0)}(x)\right)^2\right]\ud x
+ \int_0^{t^{(m)}}\int_\O \DD(\Udisc(x,t))\gradD c(x,t)\cdot\gradD c(x,t)\ud x\ud t \\
- \int_0^{t^{(m)}}\int_\O \trunc\left(\PiD c(x,t)\right)\Udisc(x,t)\cdot\gradD c(x,t)\ud x\ud t
+ \int_0^{t^{(m)}}\int_\O \left(\PiD c(x,t)\right)^2 q^P(x,t)\ud x\ud t \\
\leq \int_0^{t^{(m)}}\int_\O \hat{c}(x,t)q^I(x,t)\PiD c(x,t)\ud x\ud t.
\end{multline}
One estimates the diffusion-dispersion term from below using the coercivity \eqref{hyp:D}.
Thanks to the nonnegativty of $q^P$, the last term on the left-hand side is nonnegative.
Applying Young's inequality with $\varepsilon=\alpha_\DD$ to the convection term, we obtain
\begin{multline} \label{eq:beforesup}
\frac{1}{2}\int_\O\Phi(x)\left[\left(\PiD c^{(m)}(x)\right)^2 - \left( \PiD c^{(0)}(x)\right)^2\right]\ud x
+ \frac{\alpha_\DD}{2}\int_0^{t^{(m)}}\int_\O\left| \gradD c(x,t)\right|^2\ud x\ud t \\
\leq T\norm{q^I}_{\Leb{\infty}{2}}\norm{\PiD c}_{\Leb{\infty}{2}}
+ \frac{T}{2\alpha_\DD}\norm{\Udisc}_{\Leb{\infty}{2}}^2. 
\end{multline}
In particular, using the estimate \eqref{est:elliptic}
on $\Udisc$ and again applying Young's inequality with $\varepsilon=\frac{\phi_{\ast}}{2T}$,
\begin{multline*}
\frac{\phi_\ast}{2}\int_\O\left( \PiD c^{(m)}(x)\right)^2\ud x\\
\leq \frac{1}{2\phi_\ast}\norm{\PiD c^{(0)}}_{L^2(\O)}^2
+ \frac{\phi_\ast}{4}\norm{\PiD c}_{\Leb{\infty}{2}}^2
+ \frac{T^2}{\phi_\ast}\norm{q^I}_{\Leb{\infty}{2}}^2
+ \frac{T C^2}{2\alpha_\DD}.
\end{multline*}
The right-hand side of this inequality does not depend on $m$. 
Since
\begin{equation*}
\norm{\PiD c}_{\Leb{\infty}{2}}^2
=\sup_{m=1,\ldots,N}\int_\O\left(\PiD c^{(m)}(x)\right)^2\ud x,
\end{equation*}
this yields the estimate on $\norm{\PiD c}_{\Leb{\infty}{2}}$. The estimate on 
$\norm{\PiDe c}_{\Leb{\infty}{2}}$ follows from the fact that
$\norm{\PiDe c}_{\Leb{\infty}{2}}
= \max(\norm{\PiD c}_{\Leb{\infty}{2}}, \norm{\PiD c^{(0)}}_{L^{2}(\O)})$.
Revisiting \eqref{eq:beforesup} gives the estimate on $\gradD c$.
\end{proof}

The nonlinearity introduced by the truncation operator necessitates the use of a fixed
point argument to confirm the existence of solutions to the scheme.
% existence of solution to the gradient scheme
\begin{corollary} \label{cor:exist}
Assume \eqref{assumptions} and let $\disc$ be a gradient discretisation. Then there
exists at least one solution $(p,c)$ to the gradient scheme of Definition \ref{def:scheme}.
\end{corollary}
\begin{proof}
Take $n\in\{0,\ldots,N-1\}$ and assume that $(p^{(n)},c^{(n)})$ are given. Lemma 
\ref{lem:ellipticest} gives the existence of the solution $p^{(n+1)}$ to \eqref{scheme_p}. 
It remains to demonstrate the existence of a solution $c^{(n+1)}$ to \eqref{scheme_c}.
Define $F:\unknowns\to\unknowns$, where for $z\in\unknowns$, $v=F(z)$ is the solution to
 \begin{equation*}
	\begin{gathered}
	\int_\O\left(\Phi(x)\PiD \frac{v-c^{(n)}}{\deltat}(x)\PiD w(x)
	+ \DD(\Udisc^{(n+1)}(x))\gradD v(x)\cdot\gradD w(x)\right)\ud x\\
	- \int_\O \trunc(\PiD z(x))\Udisc^{(n+1)}(x)\cdot\gradD w(x)\ud x
	=  \frac{1}{\deltat}\int_{t^{(n)}}^{t^{(n+1)}}\int_\O\hat{c}(x)q^I(x,t)\PiD w(x)\ud x\ud t\\ 
	- \frac{1}{\deltat}\int_{t^{(n)}}^{t^{(n+1)}}\int_\O\PiD v(x) q^P(x,t)\PiD w(x)\ud x\ud t
	\quad\forall w\in\unknowns.
	\end{gathered} 
\end{equation*}
Replacing $c^{(n+1)}$ by $z$ does not change the computations of Lemma \ref{lem:parabolicest}. 
This shows the existence of $C$ not depending on $z$ such that 	
$\norm{\PiD v}_{L^{2}(\O)}\leq C$ and 
$\norm{\gradD v}_{L^{2}(\O)^d}\leq C$.
Since $F$ is continuous, apply Brouwer's fixed point theorem to conclude.
\end{proof}

In order to obtain the required level of compactness, we must estimate the so-called
{discrete time derivative}. For this we require an appropriate dual norm. To this end, define 
% dual seminorm definition 
$B_\disc = \left\{ \PiD v \,|\, v\in\unknowns\right\}\subset L^2(\O)$, 
and equip it with the norms
\begin{equation}\label{eq:defbdnormx}
\norm{u}_{B_\disc} = \inf\left\{ \norm{v}_{\disc,\para} : v\in\unknowns\mbox{ is such that }u = \PiD v\right\}\mbox{ and}
\end{equation}
\begin{equation}\label{eq:defbdnormy}
\norm{u}_{\ast,B_\disc} =\sup\left\{\int_\O\Phi(x) u(x)\PiD w(x)\ud x : 
w\in\unknowns, \, \norm{w}_{\disc,\para}=1\right\}.
\end{equation}
Note that $\norm{\cdot}_{*,B_\disc}$ is indeed a norm \cite[Section 4.2.1]{gdmbook}, and
that
\begin{equation}\label{CS:dualnorm}
\forall u\in B_\disc\,,\;\forall w\in \unknowns\,,\;\left|\int_\O \Phi(x)u (x)\PiD w(x)\ud x\right|\le \norm{u}_{\ast,B_\disc}\norm{w}_{\disc,\para}.
\end{equation}

% time derivative estimates
\begin{lemma}\label{lem:dtc}
Assume \eqref{assumptions} and let $\disc$ be a gradient discretisation.
Let $(p,c)$ be a solution to the gradient scheme for \eqref{eq:model}. 
Then there exists $C$ depending only on 
%$d$, $T$, $\O$, $\alpha_\A$, $\alpha_\DD$, $\Lambda_\DD$, $\phi_\ast$, 
$C_P\geq C_\disc$, $\norm{q^I}_{L^\infty(\O\times(0,T))}$, 
$\norm{q^P}_{L^\infty(\O\times(0,T))}$ and $C_I\geq\norm{\PiD c^{(0)}}_{L^2(\O)}$ such that
\begin{equation}\label{est:dtc}
\int_0^T \Vert \deltaD c(t)\Vert^2_{*,B_\disc}\ud t \leq C.
\end{equation}
\end{lemma}
\begin{proof}
Take $t\in(0,T)$ and $w\in\unknowns$.
Then \eqref{scheme_c} gives
\begin{align*}
\int_\O \Phi(x)&\deltaD^{(n+\frac{1}{2})}c(x)\PiD w(x)\ud x\\
\leq{}& \norm{\PiD w}_{L^2(\O)}\left(\norm{q^I}_{L^{\infty}(\O\times(0,T))} +
\norm{\PiD c}_{\Leb{\infty}{2}}\norm{q^P}_{L^\infty(\O\times(0,T))}\right)\\
&+ \norm{\gradD w}_{L^{2}(\O)^d}\left( \Lambda_\DD\norm{\gradD c^{(n+1)}}_{L^{2}(\O)^d}
+ \norm{\Udisc^{(n+1)}}_{L^2(\O)^d}\right) \\
\leq{}& C\norm{w}_{\disc,\para}
\bigg(\norm{q^I}_{L^{\infty}(\Omega\times(0,T))}
+ \norm{\PiD c}_{\Leb{\infty}{2}}\norm{q^P}_{L^{\infty}(\Omega\times(0,T))} \\
&\qquad\qquad\qquad+ \norm{\gradD c^{(n+1)}}_{L^{2}(\O)^d}
+ \norm{\Udisc^{(n+1)}}_{L^2(\O)^d}\bigg).
\end{align*}
Take the supremum over $w\in\unknowns$ with
{$\norm{ w}_{\disc,\para}=1$}, multiply by $\deltat$ and sum
over $n=0,\ldots,N-1$. The conclusion then follows from \eqref{est:parabolic}.
\end{proof}

% --- CONVERGENCE THEOREM PROOF ---
\section{Proof of the main result} \label{sec:convergence}

\subsection{Step 1: application of compactness results} \label{ssec:step1}

Estimates \eqref{est:parabolic} and \eqref{est:dtc} show that the assumptions of \cite[Theorem 4.14]{gdmbook}
are satisfied for both the explicit and the implicit reconstruction of the concentration
(strictly speaking, \cite[Theorem 4.14]{gdmbook} is based the dual norm \eqref{eq:defbdnormy}
with $\Phi=1$, but the adaptation of the proof to the case of a $\Phi$ satisfying \eqref{hyp:porosity}
is straightforward). This theorem thus provides $\cbar_1,\cbar_2\in\Leb{1}{2}$ such that, 
up to a subsequence, $\PiDm\cDm\to\cbar_1$ and $\PiDme\cDm\to\cbar_2$
in $\Leb{1}{2}$ as $m\to\infty$.
Moreover, thanks to the 
$\Leb{\infty}{2}$ estimate \eqref{est:parabolic}, these convergences also holds in
$\Leb{p}{2}$ for every $p<\infty$ and in $\Leb{\infty}{2}$ weak-$\ast$. 
We handle the case when $p=\infty$ in Step 5. 

% showing weak limits of implicit and explicit evaluations are identical
We now show that $\cbar_1=\cbar_2$. Take $\varphi\in C^\infty_c(\O\times(0,T))$.
By Lemma \ref{lem:interp}, there exists $v_m=(v_m^{(n)})_{n=0,\ldots,N_m}\in \unknownsm^{N_m+1}$
such that, as $m\to\infty$, $\PiDm v_m\to \varphi$ in $\Leb{2}{2}$ and $\gradDm v_m\to \nabla\varphi$
in $\Leb[d]{2}{2}$. Letting, by abuse of notation, $v_m(t)=v_m^{(n+1)}$ whenever $t\in (t^{(n)},t^{(n+1)}]$,
this shows in particular that $(\norm{v_m(\cdot)}_{\discm,\para})_{m\in\NN}$
is bounded in $L^2(0,T)$.
For $t\in (t^{(n)},t^{(n+1)}]$, $\PiDm\cDm(x,t)-\PiDme\cDm(x,t)=\deltatm \deltaDm\cDm(t)$ and thus,
by \eqref{CS:dualnorm},
\begin{multline*}
\left|\int_0^T\int_\O \Phi(x)\left(\PiDm\cDm(x,t)-\PiDme\cDm(x,t)\right)\PiDm v_m(x,t)\ud x\ud t \right| \\
\leq\deltatDm\int_0^T \norm{\deltaDm\cDm(t)}_{\ast,B_{\discm}}\norm{v_m(t)}_{\discm,\para}\ud t.
\end{multline*}
The estimate \eqref{est:dtc} shows that the right-hand side of this inequality
vanishes as $m\to\infty$.
Since $\left(\PiDm\cDm\right)_{m\in\NN}$ and $\left(\PiDme\cDm\right)_{m\in\NN}$
are bounded in $\Leb{2}{2}$, the convergence properties of $v_m$ yield
\begin{align*}
0 &= \lim_{m\to\infty}\int_0^T\int_\O\Phi(x)\left(\PiDm\cDm(x,t)-\PiDme\cDm(x,t)\right)\PiDm v_m(x,t)\ud x\ud t \\
&= \lim_{m\to\infty}\int_0^T\int_\O\Phi(x)\left(\PiDm\cDm(x,t)-\PiDme\cDm(x,t)\right)\varphi(x,t)\ud x\ud t,
\end{align*}
which implies that the weak-$\ast$ limits of $\left(\PiDm\cDm\right)_{m\in\NN}$ 
and $\left(\PiDme\cDm\right)_{m\in\NN}$ in $\Leb{\infty}{2}$ are identical.
We write $\cbar=\cbar_1=\cbar_2$ for this common limit.
Thanks to \cite[Lemma 4.8]{gdmbook}, the estimates \eqref{est:parabolic} on $\PiD c$ 
and $\gradD c$ imply that $\cbar$ belongs to $L^2(0,T;H^1(\O))$, and moreover
that $\gradDm\cDm\weakto\nabla\cbar$ weakly in $\Leb[d]{2}{2}$.

We turn now to the discrete pressure and discrete Darcy velocity. Estimates \eqref{est:elliptic}
give the existence of $\pbar\in\Leb{\infty}{2}$, $V\in\Leb[d]{\infty}{2}$ and 
$\Ubar\in\Leb[d]{\infty}{2}$ such that, up to a subsequence,
$\PiDm\pDm\weakto\pbar$ in $\Leb{\infty}{2}$ weak-$\ast$, 
$\gradDm\pDm\weakto V$ in $\Leb[d]{\infty}{2}$ weak-$\ast$ and
$\UDm\weakto\Ubar$ in $\Leb[d]{\infty}{2}$ weak-$\ast$. 
Then \cite[Lemma 4.8]{gdmbook} shows that $V=\nabla\pbar$ and therefore that
$\pbar\in L^\infty(0,T; H^1(\O))$. Hypothesis \eqref{hyp:A} on $\A$ and the strong
convergence of $\PiDme\cDm$ to $\cbar$ show that $\A(\PiDme\cDm)\to\A(\cbar)$ in
$\Leb[d\times d]{p}{2}$ for any $p<\infty$. Combined with the weak convergence of
$\gradDm\pDm$ to $\nabla\pbar$, using weak-strong convergence we pass to the limit
on $\UDm=-\A(\PiDme\cDm)\gradDm\pDm$ to see that $\Ubar=-\A(\cbar)\nabla\pbar$. 

% --- STEP 2: PASSING TO LIMIT ON ELLIPTIC EQUATION ---
\subsection{Step 2: convergence of the scheme for the pressure equation} \label{ssec:step2}

We first pass to the limit on \eqref{scheme_p}. Take $\varphi\in L^2(0,T;H^1(\O))$.
The interpolation result of Lemma \ref{lem:interp} yields $v_m=(v_m^{(n)})_{n=0,\ldots,N_m}\in
\unknownsm^{N_m+1}$ such that $\PiDm v_m\to \varphi$ in $L^2(0,T;L^2(\O))$
and $\gradDm v_m\to \nabla\varphi$ in $L^2(0,T;L^2(\O)^d)$ as $m\to\infty$.
Take $w=\deltatm v_m^{(n+1)}$ as a test function in the scheme for the
elliptic equation and sum on $n=0,\ldots, N_m-1$ to obtain
\begin{multline*}
-\int_0^T\int_\O\UDm(x,t)\cdot\gradDm v_m(x,t)\ud x\ud t
+\int_0^T\left(\int_\O\PiDm\pDm(x,t)\ud x\right)\left(\int_\O \PiDm v_m(x,t)\ud x\right)\ud t\\
=\int_0^T\int_\O \left( q^I - q^P\right)(x,t)\PiDm v_m(x,t)\ud x \ud t.
\end{multline*}
By weak-strong convergence, we can pass to the limit $m\to\infty$ in each of the terms above to see that
\begin{multline}\label{eq:ell_limit}
-\int_0^T\int_\O \Ubar(x,t)\cdot\nabla\varphi(x,t)\ud x\ud t 
+ \int_0^T\left(\int_\O \pbar(x,t)\ud x\right)\left(\int_\O \varphi(x,t)\ud x\right)\ud t\\
= \int_0^T\int_\O \left(q^I - q^P\right)(x,t)\varphi(x,t)\ud x\ud t.
\end{multline}
Taking $\varphi(x,t)=\theta(t)$ with $\theta\in C^\infty_c(0,T)$ shows that $\int_\O \pbar(x,t)\ud x=0$
for almost-every $t\in(0,T)$, and thus that $\pbar\in L^\infty(0,T; H^1_\star(\O))$. 
The relation \eqref{eq:ell_limit} then reduces to 
\begin{equation} \label{eq:pschemeconv}
-\int_0^T\int_\O \Ubar(x,t)\cdot\nabla\varphi(x,t)\ud x\ud t 
= \int_0^T\int_\O \left(q^I - q^P\right)(x,t)\varphi(x,t)\ud x\ud t.
\end{equation}
This has been established for $\varphi$ in $L^2(0,T;H^1(\O))$, but by density of this space in $L^1(0,T;H^1(\O))$
\eqref{eq:pschemeconv} also obviously holds for all test functions in this latter space.

% --- STEP 3: STRONG CONVERGENCE OF PRESSURE ---
\subsection{Step 3: strong convergence of the approximate pressure} \label{ssec:step3}

Analogously to the continuous problem, in order to pass to the limit on the 
diffusion-dispersion term in the discretised transport equation 
we need the strong convergence of $\UDm$ to $\Ubar$ in $\Leb[d]{2}{2}$. This begins
with the strong convergence of $\gradDm\pDm$ to $\nabla\pbar$ in the same space.

Take $w=\pDm^{(n+1)}$ in the scheme for the elliptic equation, multiply by $\deltatm$
and sum on $n=0,\ldots,N_m-1$:
\begin{multline*}
\int_0^T\int_\O \A\left(\PiDme\cDm(x,t)\right)\gradDm\pDm(x,t)\cdot\gradDm\pDm(x,t)\ud x\ud t \\
+ \int_0^T\left(\int_\O \PiDm\pDm(x,t)\ud x\right)^2\ud t
= \int_0^T\int_\O\left(q^I - q^P\right)(x,t)\PiDm\pDm(x,t)\ud x\ud t.
\end{multline*}
From the weak convergence of $\PiDm\pDm$ to $\pbar$, as $m\to\infty$ the right-hand
side converges to 
\begin{equation*}
\int_0^T\int_\O \left(q^I - q^P\right)(x,t)\pbar(x,t)\ud x\ud t
= \int_0^T\int_\O \A\left(\cbar(x,t)\right)\nabla\pbar(x,t)\cdot\nabla\pbar(x,t)\ud x\ud t,
\end{equation*}
thanks to \eqref{eq:pschemeconv} and the identification $\Ubar=-\A(\cbar)\nabla\pbar$. So
\begin{multline}
\lim_{m\to\infty}\left[\int_0^T\int_\O \A\left(\PiDme\cDm(x,t)\right)\gradDm\pDm(x,t)\cdot\gradDm\pDm(x,t)\ud x\ud t\right.\\
\left.+\int_0^T\left(\int_\O \PiDm\pDm(x,t)\ud x\right)^2\ud t \right]\\
= \int_0^T\int_\O \A\left(\cbar(x,t)\right)\nabla\pbar(x,t)\cdot\nabla\pbar(x,t)\ud x\ud t,
\label{conv.avg.pm}\end{multline}
from which we deduce that
\begin{multline} \label{eq:limsupA}
\limsup_{m\to\infty}\int_0^T\int_\O \A\left(\PiDme\cDm(x,t)\right)\gradDm\pDm(x,t)\cdot\gradDm\pDm(x,t)\ud x\ud t \\
\leq \int_0^T\int_\O \A\left(\cbar(x,t)\right)\nabla\pbar(x,t)\cdot\nabla\pbar(x,t)\ud x\ud t.
\end{multline}
The strong convergence of $\PiDme\cDm$ and hypothesis \eqref{hyp:A} show that
$\A\left(\PiDme\cDm\right)\to\A(\cbar)$ in $\Leb[d\times d]{2}{2}$, and thus
almost everywhere up to a subsequence. By dominated convergence, this shows that
$\A\left(\PiDme\cDm\right)\nabla\pbar\to\A(\cbar)\nabla \pbar$ in $\Leb[d\times d]{2}{2}$.
Then, using hypothesis \eqref{hyp:A},
\begin{align*}
\alpha_\A&\norm{\gradDm\pDm - \nabla\pbar}_{\Leb[d]{2}{2}}^2 \\
\leq{}& \int_0^T\int_\O \A\left(\PiDme\cDm(x,t)\right)\left(\gradDm\pDm - \nabla\pbar\right)(x,t) %...
\cdot\left(\gradDm\pDm - \nabla\pbar\right)(x,t)\ud x\ud t \\
={}& \int_0^T\int_\O \A\left(\PiDme\cDm(x,t)\right)\gradDm\pDm(x,t)\cdot\gradDm\pDm(x,t)\ud x\ud t \\
&- \int_0^T\int_\O\gradDm\pDm(x,t)\cdot\A\left(\PiDme\cDm(x,t)\right)\nabla\pbar(x,t)\ud x\ud t\\
&-\int_0^T\int_\O\A\left(\PiDme\cDm(x,t)\right)\nabla\pbar(x,t)\cdot(\gradDm\pDm(x,t)-\nabla\pbar(x,t))\ud x\ud t
\end{align*}
As $m\to\infty$, by weak-strong convergence the second term in the right-hand side converges to
\[
\int_0^T\int_\O\A\left(\cbar(x,t)\right)\nabla \pbar(x,t)\cdot\nabla\pbar(x,t)\ud x\ud t
\]
and the last term tends to $0$. Taking the superior limit and using \eqref{eq:limsupA}
then shows that $\gradDm\pDm\to\nabla\pbar$ strongly in $\Leb[d]{2}{2}$.
Combined with the strong convergence of $\PiDme\cDm$ and
hypothesis \eqref{hyp:A}, this implies that $\UDm\to\Ubar$ strongly in $\Leb[d]{2}{2}$.

We can now address the strong convergence of $\PiDm\pDm$, following the ideas of
Eymard et al. \cite[Lemma 5]{tp}. 
By Lemma \ref{lem:interp}, we can find $v_m=(v_m^{(n)})_{n=0,\ldots,N_m}\in \unknownsm^{N_m+1}$ such
that, as $m\to\infty$, $\PiDm v_m\to \pbar$ in $\Leb{2}{2}$ and $\gradDm v_m\to \nabla\pbar$ in
$\Leb[d]{2}{2}$.
The coercivity of $(\discm)_{m\in\NN}$ implies that
\begin{align}
\Vert\PiDm&\left(\pDm - v_m\right)\Vert^2_{\Leb{2}{2}}\nonumber\\
\leq{}& C_P\left( \int_0^T\norm{\gradDm \left(\pDm - v_m\right)(t)}^2_{L^2(\O)}\ud t
+ \int_0^T \left(\int_\O \PiDm (\pDm - v_m)(x,t)\ud x\right)^2\ud t\right)\nonumber\\
\leq{}& C_P\int_0^T\norm{\gradDm \left(\pDm - v_m\right)(t)}^2_{L^2(\O)}\ud t
+ 2C_P\int_0^T \left(\int_\O \PiDm \pDm(x,t)\ud x\right)^2\ud t\nonumber\\
&+2C_P\int_0^T \left(\int_\O \PiDm v_m(x,t)\ud x\right)^2\ud t.
\label{for.cv.pm}
\end{align}
By strong convergence of $(\gradDm \pDm)_{m\in\NN}$ and $(\gradDm v_m)_{m\in\NN}$,
the term involving the gradients tend to $0$. The strong convergence of $(\PiDm v_m)_{m\in\NN}$
and the fact that $\int_\O \pbar(x,t)\ud x=0$ for almost every $t\in (0,T)$ shows that the
last term tends to $0$. The strong convergences of $(\gradDm \pDm)_{m\in\NN}$ and
of $(\PiDme \cDm)_{m\in\NN}$, and Equation \eqref{conv.avg.pm} show that
\[
\lim_{m\to\infty}\int_0^T \left(\int_\O \PiDm \pDm(x,t)\ud x\right)^2\ud t= 0.
\]
Injected into \eqref{for.cv.pm}, these convergences show that $\PiDm \pDm\to\pbar$ strongly in $\Leb{2}{2}$.
Furthermore, thanks to estimate \eqref{est:elliptic}, this convergence holds in $\Leb{p}{2}$
for any $p<\infty$. 

% --- STEP 4: PASSING TO LIMIT ON PARABOLIC EQUATION ---
\subsection{Step 4: convergence of the scheme for the transport equation} \label{ssec:step4}

Let $\psi\in L^2(0,T;H^1(\O))$ such that $\partial_t \psi\in L^2(\O\times(0,T))$,
and take $v_m=(v_m^{(n)})_{n=0,\ldots,N_m}\in \unknownsm^{N_m+1}$
given for $\psi$ by Lemma \ref{lem:interp}. As $m\to\infty$, setting
$\gradDme v_m(x,t)=\gradDm v_m^{(n)}(x)$ for all $t\in (t^{(n)},t^{(n+1)}]$ and all $n=0,\ldots, N_m-1$,
\begin{align*}
&\PiDme v_m\to\psi\mbox{ in $\Leb{2}{2}$, }\gradDme v_m\to\nabla\psi\mbox{ in $\Leb[d]{2}{2}$, }\\
&\PiDm v_m^{(0)}\to \psi(\cdot,0)\mbox{ in $L^2(\O)$, and }\deltaDm v_m\to\partial_t \psi\mbox{ in $\Leb{2}{2}$}.
\end{align*}
Take $w=\deltatm v_m^{(n)}$ as a test function in the scheme for the 
concentration equation and sum on $n=0,\ldots, N_m-1$. We obtain
$\cS_1^{(m)} + \cS_2^{(m)} - \cS_3^{(m)} = \cS_4^{(m)} - \cS_5^{(m)}$,
where
\begin{gather*}
\cS_1^{(m)} = \sum_{n=0}^{N_m-1} \int_\O\Phi(x)\left( \PiDm\cDm^{(n+1)} - \PiDm\cDm^{(n)}\right)\PiDm v_m^{(n)}(x)\ud x\ud t, \\
\cS_2^{(m)} = \int_0^T\int_\O \DD\left(\UDm(x,t)\right)\gradDm\cDm(x,t)\cdot\gradDme v_m(x,t)\ud x\ud t, \\
\cS_3^{(m)} = \int_0^T\int_\O \trunc(\PiDm\cDm(x,t))\UDm(x,t)\cdot\gradDme v_m(x,t)\ud x\ud t, \\
\cS_4^{(m)} = \int_0^T\int_\O\left(\hat{c}q^I\right)(x,t)\PiDme v_m(x,t)\ud x \ud t, \mbox{ and}\\
\cS_5^{(m)} = \int_0^T\int_\O \left(\PiDm\cDm q^P\right)(x,t)\PiDme v_m(x,t)\ud x \ud t.
\end{gather*}
Using a discrete integration-by-parts \cite[Appendix D.1.7]{gdmbook}, the terms 
$[ \PiDm\cDm^{(n+1)} - \PiDm\cDm^{(n)}]\PiDm v_m^{(n)}$ appearing in $\cS_1^{(m)}$
can be transformed into $[\PiDm v_m^{(n)}-\PiDm v_m^{(n+1)} ]\PiDm\cDm^{(n+1)}$, so that
\[
\cS_1^{(m)} = -\int_0^T\int_\O \Phi(x)\PiDm\cDm(x,t) \deltaDm v_m(x,t)\ud x\ud t
- \int_\O \Phi(x)\PiDm\cDm^{(0)}(x)\PiDm v_m^{(0)}(x)\ud x.
\]
The consistency of $(\discm)_{m\in\NN}$ gives $\PiDm\cDm^{(0)}=\PiDm\interpDm c_0\to c_0$ in $L^2(\O)$. 
The convergence properties of $\deltaDm v_m$, $\PiDm v_m^{(0)}$ and $\PiDm \cDm$ then show that, as $m\to\infty$,
\[
\cS_1^{(m)}\to -\int_0^T\int_\O \Phi(x)\cbar(x,t) \partial_t \psi(x,t)\ud x\ud t
- \int_\O \Phi(x)c_0(x)\psi(x,0)\ud x.
\]
By convergence of $\gradDme v_m$, the convergence of $\cS_2^{(m)}$ is assured if we show that 
$\DD\left(\UDm\right)\gradDm\cDm \weakto \DD(\Ubar)\nabla\cbar$ weakly in $\Leb[d]{2}{2}$.
To this end, the strong convergence of $\UDm$ and hypothesis \eqref{hyp:D} show that 
$\DD(\UDm)\to\DD(\Ubar)$ in $\Leb[d\times d]{2}{2}$. Together with the weak convergence
of $\gradDm\cDm$ and the bound on 
$\left(\DD(\UDm)\gradDm\cDm\right)_{m\in\NN}$ in $\Leb[d]{2}{2}$, by Lemma \ref{lem:ws},
$\DD(\UDm)\gradDm\cDm\weakto\DD(\Ubar)\nabla\cbar$ in this space, thus ensuring the
convergence of $\cS_2^{(m)}$.
The convergence of $\cS_3^{(m)}$ is completely analogous. 
Terms $\cS_4^{(m)}$ and $\cS_5^{(m)}$ are straightforward with the convergences of $\PiDme v_m$
and $\PiDm \cDm$. Therefore let $m\to\infty$ in $\cS_1^{(m)} + \cS_2^{(m)} - \cS_3^{(m)} = \cS_4^{(m)} - \cS_5^{(m)}$ to obtain
\begin{equation}\label{transp.eq}
\begin{aligned}
&-\int_0^T\int_\O \Phi(x)\cbar(x,t)\partial_t\psi(x,t)\ud x\ud t
- \int_\O \Phi(x) c_0(x)\psi(x,0)\ud x \\
&+ \int_0^T\int_\O\DD(\Ubar(x,t))\nabla\cbar(x,t)\cdot\nabla\psi(x,t)\ud x\ud t 
- \int_0^T\int_\O \trunc(\cbar(x,t))\Ubar(x,t)\cdot\nabla\psi(x,t)\ud x\ud t \\
&\qquad= \int_0^T\int_\O\left(\hat{c}q^I\right)(x,t)\psi(x)\ud x \ud t
- \int_0^T\int_\O\left(\cbar q^P\right)(x,t)\psi(x)\ud x \ud t.
\end{aligned}
\end{equation}
This relation has been proved for any test function in the space
$\cT = \{ \psi\in L^2(0,T;H^1(\O))\,:\,\partial_t\psi\in \Leb{2}{2}\}$. Since this space is dense
$L^2(0,T; H^1(\O))$, to show that $\Phi\partial_t\cbar$ belongs to 
$L^2(0,T; (H^1(\O))')= (L^2(0,T; H^1(\O))'$, it suffices to show that the linear form
\begin{equation*}
\cT\to\RR,\ \psi\mapsto\langle \Phi\partial_t\cbar,\psi\rangle_{\distt',\mathcal{D}}
= - \int_0^T\int_\O \Phi(x)\cbar(x,t)\partial_t\psi(x,t)\ud x\ud t
\end{equation*}
is continuous for the $L^2(0,T; H^1(\O))$ norm. To see this, transform \eqref{transp.eq}
to write the integral involving $\partial_t\psi$ in terms of integrals purely involving
spatial derivatives, and use the estimates \eqref{est:elliptic} on $\Udisc$ and \eqref{est:parabolic}
together with the regularity of $\hat{c}$ and the sources.

Moreover, because $\Phi$ is independent of time, we have
$\partial_{t}(\Phi\cbar)=\Phi\partial_t\cbar\in L^2(0,T;H^1(\O))$. Since $\Phi\cbar\in L^2(0,T;\Phi H^1(\O))$,
where $\Phi H^1(\O) := \{ \Phi u : u\in H^ 1(\O)\}$ embeds
compactly into $L^2(\O)$. This implies \cite[Section 2.5.2]{poly} that 
$\Phi\cbar$ can be identified with an element of $C([0,T];L^2(\O))$ with, thanks to \eqref{transp.eq}, the property
that $\Phi\cbar(\cdot,0)=\Phi c_0$ in $L^2(\O)$. Again, since $\Phi$ does not depend upon time
we then have $\cbar\in C([0,T];L^2(\O))$ and $\cbar(\cdot,0)=c_0$ in $L^2(\O)$.

Integrating-by-parts the first term in \eqref{transp.eq} and using the density of $\cT$ in
$L^2(0,T; H^1(\O))$, we see that the transport equation in \eqref{eq:weaksol} is satisfied, except with the truncation operator $\trunc$ applied to $\cbar$
in the convection term. Showing that the equation is satisfied without this operator
amounts to establishing the estimate $0\leq\cbar(x,t)\leq1$ for almost-every $(x,t)\in\O\times(0,T)$.
Fabrie and Gallou\"et \cite[Proposition 4.1]{fg00} prove precisely this estimate in our setting
of a bounded diffusion-dispersion tensor, and with a function $f$ that plays the role
of our $\trunc$.

To summarise, we have shown that the limit $(\pbar,\cbar)$ of $(\cDm,\pDm)$ is a solution
of \eqref{eq:model} in the sense of Definition \ref{def:weaksol}, and that properties
(i), (ii) and (iv) of Theorem \ref{th:main} hold. It remains to address (iii), the
uniform-in-time, strong-$L^2(\O)$ convergence of the approximate concentration.

\begin{remark}\label{rem:DDbounded} 
The boundedness hypothesis on $\DD$ was solely used to prove that a
solution to \eqref{transp.eq} remains between $0$ and $1$.
The proof of this in \cite{fg00} relies on using the negative part of $c$ and positive parts of $(1-c)$ as
test functions, which is made possible for $\DD$ bounded because solutions
and test functions are both be taken in $L^2(0,T;H^1(\O))$.
If $\DD$ is given by \eqref{eq:ddt}, then test functions must be taken in $L^2(0,T;W^{1,4}(\O))$
whereas the solution is still only in $L^2(0,T;H^1(\O))$ \cite{dt14}, and proving
that a solution to \eqref{transp.eq} remains between $0$ and $1$ is an open
problem. Should this problem be solved, our convergence analysis would
apply with minor modifications to $\DD$ given by \eqref{eq:ddt}.
\end{remark}

% --- STEP 5: UNIFORM-IN-TIME CONVERGENCE OF CONCENTRATION ---
\subsection{Step 5: $\Leb{\infty}{2}$ convergence of the approximate concentration} \label{ssec:step5}

Fix $T_0\in[0,T]$ and take a sequence $(\Tm)_{m\in\NN}\subset[0,T]$ with $\Tm\to T_0$
as $m\to\infty$. Denote by $k_m\in\{0, \ldots, N_m-1\}$ the index such that 
$\Tm\in(t^{(\km)}, t^{(\km+1)}]$. Apply the uniform-in-time, weak-in-space compactness result of \cite[Theorem 4.19]{gdmbook} with estimates \eqref{est:parabolic} and \eqref{est:dtc}
to obtain $\PiDm\cDm\to\cbar$ and $\PiDme\cDm\to\cbar$, both in $L^\infty(0,T;L^2(\O)\weak)$,
where $L^2(\O)\weak$ denotes $L^2(\O)$ equipped with the weak topology. This gives 
$\sqrt{\Phi}\PiDm\cDm(\cdot,\Tm)\weakto\sqrt{\Phi}\cbar(\cdot,T_0)$ weakly in $L^2(\O)$ and hence
\begin{equation} \label{eq:liminfc}
\liminf_{m\to\infty}\int_\O \Phi(x)\left( \PiDm\cDm(x,\Tm)\right)^2\ud x
\geq \int_\O \Phi(x)\left(\cbar(x,T_0)\right)^2\ud x.
\end{equation}
Take $w=\cDm^{(n+1)}$ in the scheme, multiply by $\deltatm$ and sum over $n=0, \ldots, \km$.
Reasoning as in the proof of Lemma \ref{lem:parabolicest}, we obtain \eqref{eq:2ndenergy}
with $\discm$ and $\km$ in place of $\disc$ and $m$. 
Using the consistency of $(\discm)_{m\in\NN}$, take the limit
superior of \eqref{eq:2ndenergy} as $m\to\infty$:
\begin{align*}
\frac{1}{2}\limsup_{m\to\infty}&\int_\O\Phi\left(\PiDm\cDm(x,\Tm)\right)^2\ud x\\
\leq{}& \frac{1}{2}\int_\O \Phi(x)\left(c_0(x)\right)^2 \ud x
+ \limsup_{m\to\infty}\int_0^{t^{(\km+1)}}\int_\O \left(\hat{c}q^I\PiDm\cDm\right)(x,t)\ud x\ud t \\
&- \liminf_{m\to\infty}\int_0^{t^{(\km+1)}}\int_\O \left(\PiDm\cDm(x,t)\right)^2 q^P(x,t)\ud x\ud t\\
&+ \limsup_{m\to\infty}\int_0^{\Tm}\int_\O \trunc\left(\PiDm\cDm(x,t)\right)\UDm(x,t)\cdot\gradDm\cDm(x,t)\ud x\ud t \\
&- \liminf_{m\to\infty}\int_0^{\Tm}\int_\O \DD(\UDm(x,t))\gradDm\cDm(x,t)\cdot\gradDm\cDm(x,t)\ud x\ud t \\
=:{}& \frac{1}{2}\int_\O \Phi(x)\left(c_0(x)\right)^2 \ud x
+ \limsup_{m\to\infty}\fS_1^{(m)} - \liminf_{m\to\infty}\fS_2^{(m)}
+ \limsup_{m\to\infty}\fS_3^{(m)} - \liminf_{m\to\infty}\fS_4^{(m)}.
\end{align*}
Note that the sequence $\left(t^{(\km)}\right)_{m\in\NN}$ converges to $T_0$ as $m\to\infty$.
Since $\PiDm\cDm\to\cbar$ strongly in
$\Leb{2}{2}$,
\begin{equation*}
\limsup_{m\to\infty}\fS_1^{(m)} = \int_0^{T_0}\int_\O\left(\hat{c}q^I\cbar\right)(x,t)\ud x\ud t\quad\mbox{and}\quad
\liminf_{m\to\infty}\fS_2^{(m)} = \int_0^{T_0}\int_\O\left(\cbar(x,t)\right)^2 q^P(x,t)\ud x\ud t.
\end{equation*}
Now $\trunc\left(\PiDm\cDm\right)\to\trunc(\cbar)$ strongly in $\Leb{2}{2}$ and 
$\gradDm\cDm\weakto\nabla\cbar$ weakly in $\Leb[d]{2}{2}$. Since the product of
these two sequences is bounded in $\Leb[d]{2}{2}$, by Lemma \ref{lem:ws} we have
$\trunc\left(\PiDm\cDm\right)\gradDm\cDm\weakto\trunc(\cbar)\nabla\cbar$ weakly in
$\Leb[d]{2}{2}$. Combined with the fact that $\one_{[0,\Tm]}\UDm$ converges to $\one_{[0,T_0]}\Ubar$
strongly in $\Leb[d]{2}{2}$, we have
\begin{equation*}
\limsup_{m\to\infty}\fS_3^{(m)} = \int_0^{T_0}\int_\O \trunc(\cbar(x,t))\Ubar(x,t)\cdot\nabla\cbar(x,t)\ud x\ud t
= \int_0^{T_0}\int_\O\cbar(x,t)\Ubar(x,t)\cdot\nabla\cbar(x,t)\ud x\ud t,
\end{equation*}
since $0\leq\cbar(x,t)\leq1$ almost-everywhere on $\O\times(0,T)$.
Using arguments similar to those above, it is straightforward to verify that 
$\DD^{1/2}(\UDm)\gradDm\cDm\weakto\DD^{1/2}(\Ubar)\nabla\cbar$ weakly in $\Leb[d]{2}{2}$.
It then follows that
\begin{equation*}
\liminf_{m\to\infty}\fS_4^{(m)} \geq \int_0^{T_0}\int_\O\DD(\Ubar(x,t))\nabla\cbar(x,t)\cdot\nabla\cbar(x,t)\ud x\ud t.
\end{equation*}
Putting it all together, we have
\begin{equation}\label{liminf.energy}
\begin{aligned}
\frac{1}{2}\limsup_{m\to\infty}&\int_\O\Phi(x)\left(\PiDm\cDm(x,\Tm)\right)^2\ud x
\leq \frac{1}{2}\int_\O \Phi(x)\left(c_0(x)\right)^2 \ud x \\
&+ \int_0^{T_0}\int_\O\left(\hat{c}q^I\cbar\right)(x,t)\ud x\ud t 
- \int_0^{T_0}\int_\O\left(\cbar(x,t)\right)^2 q^P(x,t)\ud x\ud t\\
&+ \int_0^{T_0}\int_\O \cbar(x,t)\Ubar(x,t)\cdot\nabla\cbar(x,t)\ud x\ud t \\
&- \int_0^{T_0}\int_\O\DD(\Ubar(x,t))\nabla\cbar(x,t)\cdot\nabla\cbar(x,t)\ud x\ud t.
\end{aligned}
\end{equation}
{Writing the first equation of \eqref{eq:weaksol} with $\varphi=\cbar^2/2$ shows that
\begin{align*}
\int_0^{T_0}\int_\O \cbar(x,t)\Ubar(x,t)\cdot\nabla\cbar(x,t)\ud x\ud t={}&
\int_0^{T_0}\int_\O \Ubar(x,t)\cdot\nabla\left(\frac{\cbar(x,t)^2}{2}\right)\ud x\ud t\\
={}&
-\frac{1}{2}\int_0^{T_0}\int_\O \left( q^I - q^P\right)(x,t)\cbar(x,t)^2\ud x\ud t.
\end{align*}
Plugging this relation in \eqref{liminf.energy} and recalling the energy identity \eqref{eq:energyid}, we infer}
\begin{equation} \label{eq:limsupc}
\limsup_{m\to\infty}\int_\O\Phi\left(\PiDm\cDm(x,\Tm)\right)^2\ud x
\leq \int_\O \Phi(x)\left(\cbar(x,T_0)\right)^2\ud x.
\end{equation}
Comparing \eqref{eq:liminfc} and \eqref{eq:limsupc} shows that
$\lim_{m\to\infty}\Vert\sqrt{\Phi}\PiDm\cDm(\cdot,\Tm)\Vert^2_{L^2(\O)}=\Vert\sqrt{\Phi}\cbar(\cdot,T_0)\Vert^2_{L^2(\O)}$.
Together with the weak-$L^2(\O)$ convergence established earlier, this gives
$\sqrt{\Phi}\PiDm\cDm(\cdot,\Tm)\to\sqrt{\Phi}\cbar(\cdot,T_0)$ strongly in $L^2(\O)$. 
From the characterisation of uniform convergence given in Lemma \ref{lem:equiv-unifconv} 
and the uniform positivity of $\Phi$, we conclude that
$\PiDm\cDm\to\cbar$ in $\Leb{\infty}{2}$. Since $\PiDm\cDm$ is piecewise constant in time,
this convergence is actually uniform, not just ``uniform almost everywhere''. It remains to establish this convergence
for $\left(\PiDme\cDm\right)_{m\in\NN}$. To this end, observe that for every $m\in\NN$
there is a mapping $\gamma_m:\RR\to\RR$ such that $\left|\gamma_m - \Id\right|\leq\deltatm$
and $\PiDme\cDm=\PiDm\cDm\circ\gamma_m$. Then
\begin{align*}
\sup_{t\in [0,T]}\norm{\PiDme\cDm-\cbar}_{L^2(\O)}
&\le \sup_{t\in [0,T]}\norm{\PiDm\cDm\circ\gamma_m-\cbar\circ\gamma_m}_{L^2(\O)} 
+ \sup_{t\in [0,T]}\norm{\cbar\circ\gamma_m-\cbar}_{L^2(\O)} \\
& \le \sup_{t\in [0,T]}\norm{\PiDm\cDm-\cbar}_{L^2(\O)} + \sup_{t\in [0,T]}\norm{\cbar\circ\gamma_m-\cbar}_{L^2(\O)},
\end{align*}
and the second of these terms vanishes since $\cbar\in C([0,T];L^2(\O))$ and 
$\gamma_m\to\Id$ uniformly as $m\to\infty$. This completes the proof of
Theorem \ref{th:main}. \qed

% --- NUMERICAL EXPERIMENTS ---
\section{{Numerical experiments}} \label{sec:experiments}

We present here numerical results using schemes A and B defined in Section \ref{ssec:examples}. These tests show that the methods behave well when the molecular diffusion is strong enough, or for large time steps, but that they become unstable and inaccurate for small or vanishing molecular diffusion and small time steps. Upstream versions of these schemes do correct these issues, but we propose a modification that enables us to recover stable and accurate schemes at all considered time steps, and for any level of molecular diffusion.

In all tests cases, $\hat c\equiv 1$ and the viscosity $\mu$ is defined by 
\begin{equation} \label{eq:viscosity}
\mu(c) = \mu(0)\left(1 + \big(M^{1/4} - 1\big)c\right)^{-4} \quad\mbox{for $c \in [0, 1]$,} 
\end{equation}
with $\mu(0)=1$ cp and a mobility ratio $M = \mu(0)/\mu(1)$ specified for each test.

Scheme A is tested on Cartesian meshes with $25\times 25$, $50\times 50$ and $100\times 100$ cells.
Scheme B is tested on triangular meshes from the FVCA8 Benchmark \cite{2Dbench}; these meshes
are built by scaling and reproducing a certain number of times an initial triangulation of the unit square: Mesh4, shown in Figure \ref{fig:mesh4}, corresponds to 16 reproductions of this pattern,
Mesh5 to 32 reproductions, and Mesh 6 to 64 reproductions.

\begin{figure}[h!]
\resizebox{0.5\linewidth}{!}{\includegraphics{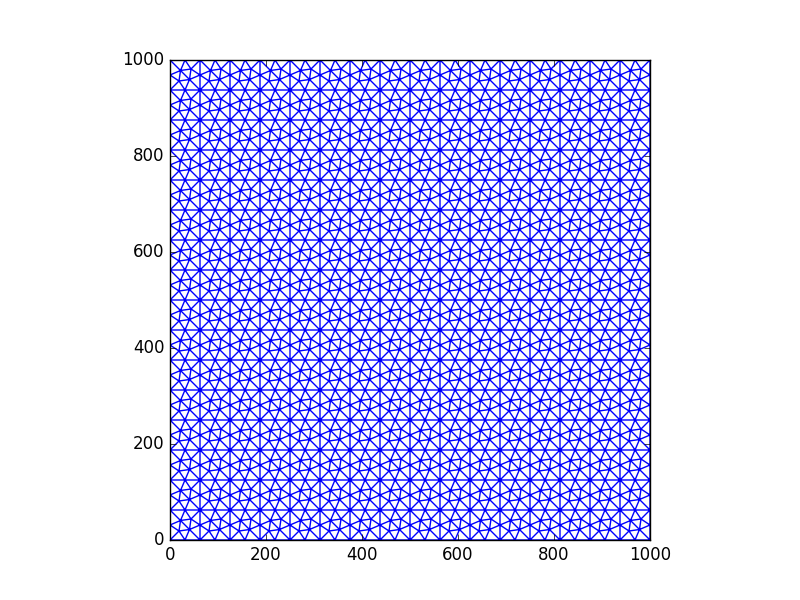}}
\caption{Mesh4}
\label{fig:mesh4}
\end{figure}

In our tests we are led to compare the initial scheme with the scheme obtained by the following modification of the diffusion--dispersion tensor, which consists in rewriting \eqref{eq:ddt} as
\[
(\D{x}{\U})_{i,i} = D_m + |\U|\Big( (D_l-D_t)\frac{\U_i^2}{|\U|^2} + D_t \Big),
\]
setting $D_\alpha = \Phi d_\alpha$ for $\alpha = m,l,t$, and
\[
(\D{x}{\U})_{i,j} = |\U|\Big( (D_l-D_t)\frac{\U_i\U_j}{|\U|^2} \Big)\hbox{ for } j\neq i.
\]
Then, for some tests with too low values for $d_m$, $d_l$ or $d_t$, we use the centred scheme in which we replace $\DD$ with ${\DD}_h$ defined by
\begin{equation} \label{eq:ddtmod}
\begin{array}{ll}
({\DD}_h(x,\U))_{i,i} = \max \Big( (\D{x}{\U})_{i,i},|\U| h \Big),\\
({\DD}_h(x,\U))_{i,j} = (\D{x}{\U})_{i,j}\hbox{ for }j\neq i,
\end{array}
\end{equation}
where $h$ is the size of the mesh (note that the dimension of $D_l$ and $D_t$ is that of a length). This modification introduces some diffusion, scaled by the approximate Darcy velocity and vanishing with the mesh size. This numerical diffusion is not larger than the diffusion induced by upstream schemes, but has the added advantage of being isotropic when the numerical diffusion from upstreaming is larger in the direction of the flow.
Note that the convergence analysis carried out in the previous sections also applies to $\DD_h$ provided
that $d_m$, $d_l$ and $d_t$ are all strictly positive, since for $h\le \min(D_l,D_t)$ we have $\DD_h=\DD$.

\subsection{Analytical solutions}

Exact solutions of nonlinear models involved in groundwater flows are very rare,
and often limited to 1D models or to single equations; see e.g.\ \cite{broadbridge1,broadbridge2}.
These do not account for the coupling occurring for example in \eqref{peacemanmodel}, and which is at the core of numerical issues (a poor approximation of the Darcy velocity reflects strongly on the concentration fields, which in return further degrades the Darcy velocity).
In this section, we design an analytic solution for a slightly modified version of
\eqref{peacemanmodel}. Specifically, we remove the source and reaction terms in the concentration equation, and we consider Dirichlet boundary conditions on part of the domain. The solution however
has similar features to those found in real solutions: it flows from one part of the domain to the other part, and satisfies the same strongly coupled equations inside the domain -- and thus enable us to illustrate corresponding numerical challenges.

All data in this section are dimensionless. We take $\Phi=1$, $c_0=0$, $\K={\rm Id}$, $\Omega=(0,1)^2$, and final time  $T = 0.4$. Following ideas in \cite{eym:grid.orientation}, we seek a solution $(p,c)$ to \eqref{peacemanmodel} in radial coordinates $\rho$ centred at $(1,1)$. The tensor $\DD$ is defined by \eqref{eq:ddt} with $\dl=\dt=0$ and $\dm>0$ chosen such that 
\begin{equation}\label{def:N}
N = \frac {2} {4 \dm} - 1\in\NN.
\end{equation}
Removing the production reaction term for the equation on $c$ (see Remark \ref{rem:sources} about the injection source term), this solution $(p,c)$ satisfies:
\begin{align*}
&u_\rho=\frac {1} {\rho} =  -\frac{1}{\mu(c(\rho,t))}\partial_\rho p(\rho,t), \quad(\rho,t)\in(0,\infty)\times(0,T),\\
&\rho \partial_{t}c(\rho,t)-\partial_\rho\left(\rho (\dm \partial_\rho c - u_\rho c)\right)(\rho,t) = 0,
\quad(\rho,t)\in(0,\infty)\times(0,T),
\end{align*}
with $c(0,t) = 1$ and $c(+\infty,t) = 0$. The full Darcy velocity is $\U(\rho,t)=u_\rho\mathbf{e}_\rho$, where $(\mathbf{e}_\rho,\mathbf{e}_\theta)$ are the local polar unit vectors.

Denoting by $\delta_A$ the Dirac measure at a point $A\in\overline{\O}$, we set
$q^I(x)\ud x= \frac{\pi}{2} \ \delta_{(1,1)}$. For $x=(x_1,x_2)\in (0,1)\times\{0\}$,  $q^P(x) = {\rm d}\theta(x_1)$ is the lineic Dirac measure weighted by the derivate of the angle $\theta(x_1)=OIM$ 
with $O=(0,0)$, $I=(1,1)$ and $M=(x_1,0)$. For $x=(x_1,x_2)\in \{0\}\times(0,1)$,  $q^P(x) = {\rm d}\theta(x_2)$ is the lineic Dirac measure weighted by the derivate of the angle $\theta(x_2)=OIM$ with $M=(0,x_2)$. Hence, solvent is injected at the top right corner of the domain and a mixture of oil
and solvent is recovered at both sides of the domain passing by the origin, both at a rate of $\frac{\pi}{2}$. These source terms, the expression for the function $\mu$, and the values for $\dm$ are inspired by Tests 1 and 2 in \cite{wan00,cd07}, in which solvent is injected at the top right corner and the mixture is produced at the bottom left corner.

\begin{remark}[Hidden source terms]\label{rem:sources}
The Darcy velocity $\U$ defined above satisfies $\dive(\U)=0$ in the sense of distributions on $\Omega$. However, one can easily check
that if $\varphi\in C^1(\RR^2)$ then, with $B_\epsilon$ the ball of center $(1,1)$ and radius $\epsilon$,
\begin{equation}\label{pv.int}
\mathrm{PV}\int_\Omega \U\cdot\nabla\varphi:=\lim_{\epsilon\to 0}\int_{\Omega\backslash B_\epsilon} \U\cdot\nabla\varphi = \langle q^I,\varphi\rangle
-\langle q^P,\varphi\rangle,
\end{equation}
where the duality products are understood in the sense of distributions on $\RR^2$ (or as integrals against the measures $q^I$ and $q^P$). Here, PV denotes the `principal value' of the integral, a classical notion in the context of distributions (but slightly less classical here since the singularity is on the boundary of the domain). Equation \eqref{pv.int} shows that the relation $\dive(\U)=q^I-q^P$ is satisfied in a weak sense against test functions in $C^1(\overline{\Omega})$.

One could also consider that the source terms $q^I$ and $q^P$ are handled as non-homogeneous Neumann boundary conditions (which is how they are implemented in our tests). Indeed, for elliptic PDEs with Neumann boundary conditions, including measures in the boundary conditions or the same measures as `source terms' of the PDE lead to the exact same weak formulations.

The same reasoning applies to the equation on $c$. Although the equation does not seem to contain any injection source term $q^I$, this source term is actually hidden `on the boundary'. 
\end{remark}

The equation on $c$ thus no longer depends on $p$, and a solution is obtained by setting $c(\rho,t) = \psi(\rho^2/(4 \dm t))$, where $\psi$ satisfies
\[
 \forall z>0,\quad -z \psi'(z) + N \psi'(z) - z \psi''(z) = 0,
\]
with $\psi(0) = 1$ and $\psi(+\infty) = 0$ and $N$ defined by \eqref{def:N}. This function is given by 
\[
 \psi(z) = e^{-z} \sum_{k=0}^N \frac {z^k} {k!}.
\]
We impose the nonhomogeneous Dirichlet boundary condition $c(x,t) = \psi(\rho^2/(4 \dm t))$ on $[(0,1)\times\{0\}]\cup  [\{0\}\times(0,1)]$, to preserve the radial symmetry of the problem and the expression of the exact solution.

\begin{remark}
The values of $\psi$ can easily be evaluated by calculating the terms $v_0 = 0$, $v_{k+1} = \frac {z} {N-k} (v_k + e^{-z} )$, for $k=0,\ldots, N-1$, and by setting $\psi(z) =  v_N+e^{-z} $.
\end{remark}

The exact solutions corresponding to these tests are shown in Figure \ref{fig:analytical.exact}. Note that in this particular situation, the values for $c$ do not depend on the function $\mu$; however, the numerical issues to approximate this solution remain since we consider the full coupled system, in which $\mu$ strongly impacts the approximation of the Darcy velocity.

\begin{figure}[!ht] 
\begin{center}
\resizebox{0.6\textwidth}{!}{\includegraphics{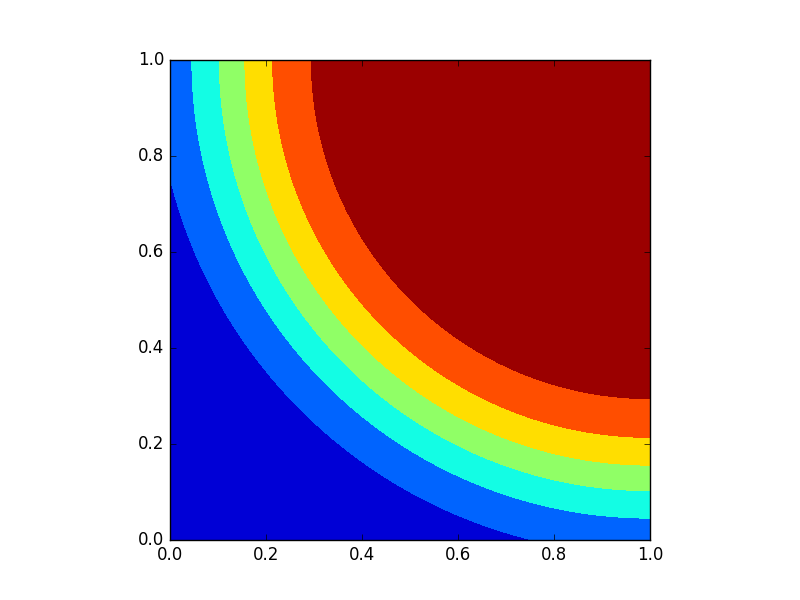}\includegraphics{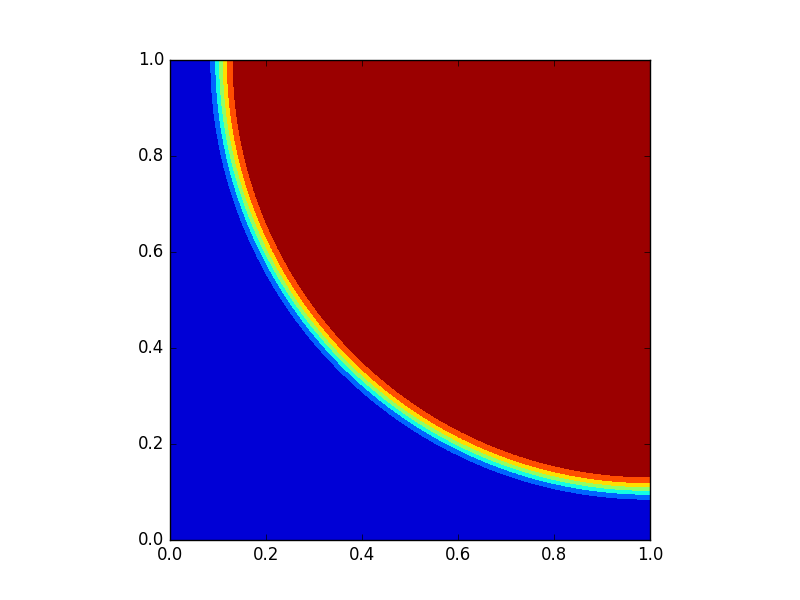}}
\end{center}
\caption{Exact solutions: analytical test 1 (left) corresponds to $d_m=0.05$;
analytical test 2 (right) corresponds to $d_m=0.001$.
Values from $0$ (dark blue) to $1$ (dark red).}
\protect{\label{fig:analytical.exact}}
\end{figure}

\subsubsection{Analytical test 1: constant viscosity, large $d_m$}

We take here $M=1$ (hence, $\mu$ is constant) and $d_m = 0.05$. 
Figure \ref{fig:scheme_a_b_ta1} presents the concentration $c$ calculated by the two schemes; $L^1$ and $L^2$ errors for various meshes and time steps are shown in
Table \ref{tab:ta1}. All these results indicate a clear convergence of both schemes,
with a rate close to $\mathcal O(h^2+\delta t)$, as expected for the first order Schemes A and B
with implicit time stepping.

\begin{figure}[!ht] 
\begin{center}
\resizebox{0.6\textwidth}{!}{\includegraphics{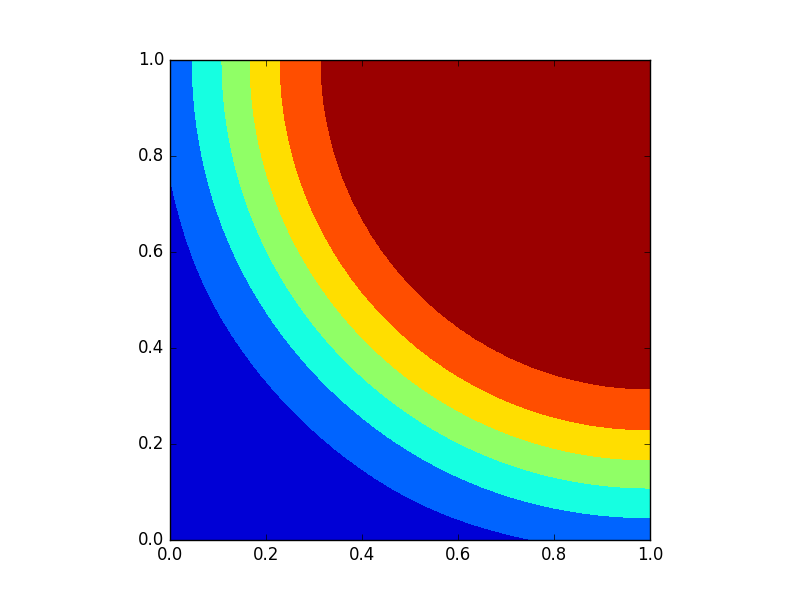}\includegraphics{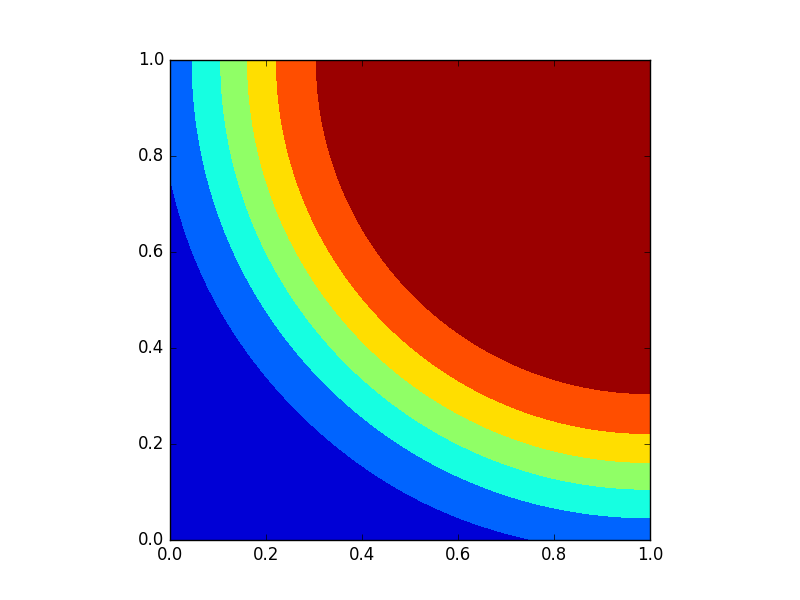}}
\end{center}
\caption{Analytical test 1, concentration with centred scheme: Scheme A on a $50\times50$ mesh with time step $0.01$ (left); Scheme B on Mesh5 with time step $0.005$ (right).
Values from $0$ (dark blue) to $1$ (dark red).}
\protect{\label{fig:scheme_a_b_ta1}}
\end{figure}

\begin{table}
\begin{tabular}{|c|c|c|c|c|}
\hline
Scheme & Mesh & $\delta t$ & $L^1$ error & $L^2$ error\\
\hline
\hline
\multirow{3}{*}{A} &   $25\times 25$ & $0.02$ & 2.38E-2  &  3.23E-2\\
\cline{2-5}
 &$50\times 50$ & $0.005$ & 6.69E-3 & 9.10E-3 \\
\cline{2-5}
 & $100\times 100$ & $0.00125$ & 1.73E-3 & 2.36E-3\\
\hline
\hline
\multirow{3}{*}{B} & mesh4 & $0.02$ & 2.39E-2 & 3.20E-2\\
\cline{2-5}
 & mesh5 & $0.005$ & 6.70E-3 & 9.04E-3\\
\cline{2-5}
 & mesh6 & $0.00125$ & 1.73E-3 & 2.38E-3\\
\hline
\end{tabular}
\caption{Analytical test 1: $L^1$ and $L^2$ errors at the final time $T=0.4$.}
\label{tab:ta1}
\end{table}

\subsubsection{Analytical test 2: varying viscosity, small $d_m$}

In this test, we let $M=40$ and $d_m = 0.001$. The viscosity therefore varies, the natural diffusion is very small, and the problem is highly advection-dominated. As can be seen in Figures \ref{fig:scheme_a_ta2}--\ref{fig:scheme_b_ta2} and in Tables \ref{tab:ta2_a}--\ref{tab:ta2_b}, the
centred versions of Schemes A and B produce unacceptable solutions, and do not seem to converge.
The same conclusion holds for upstream versions of these schemes: even though the grid effects seem
to be somehow mitigated by the upstreaming, the solutions are still very distorted and there
is no apparent numerical convergence. 

The reasons for this failure of both the centred and upstream schemes might be found in the combination of two factors: the viscosity varying with $c$ negatively impacts the quality of the numerical Darcy velocity, which in turns generate bad fluxes when used in the concentration equation; in case of upstreaming, the numerical diffusion introduced by this process is anisotropic. It occurs mostly in the direction of the numerical Darcy velocity, and a poorly approximated velocity therefore results in numerical diffusion in unphysical directions -- typically, directions dictated by the grid rather than the genuine flow.

To mitigate these issues, we use the modification \eqref{eq:ddtmod}.
The results presented in Figures \ref{fig:scheme_a_ta2}--\ref{fig:scheme_b_ta2} and in Tables \ref{tab:ta2_a}--\ref{tab:ta2_b} show a clear improvement of the numerical solution. It has the expected
shape, and convergence seems to occur (even though at a slow rate).

\begin{figure}[!ht] 
\begin{center}
\resizebox{0.6\textwidth}{!}{\includegraphics{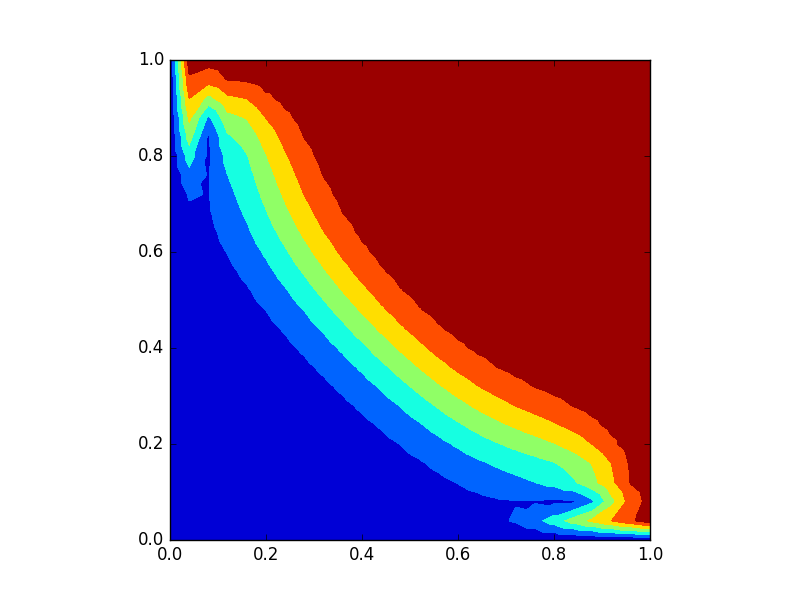}\includegraphics{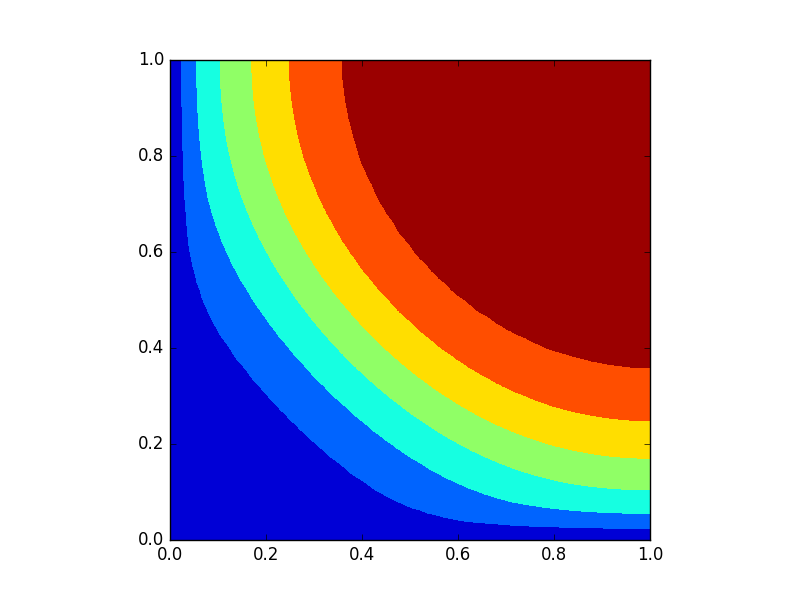}\includegraphics{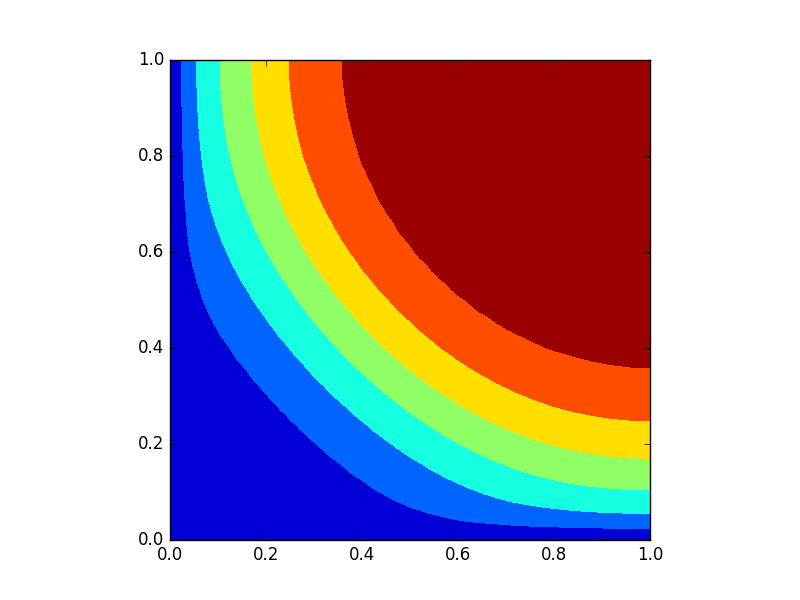}}
\resizebox{0.6\textwidth}{!}{\includegraphics{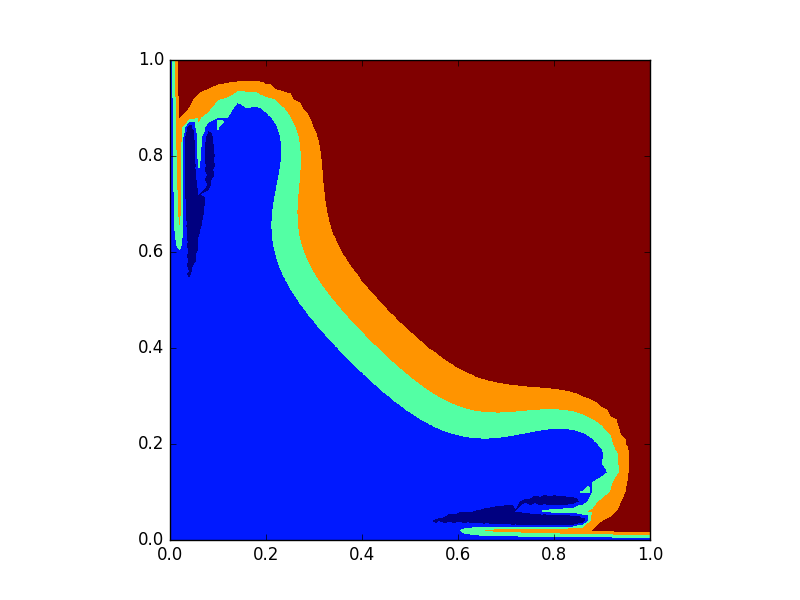}\includegraphics{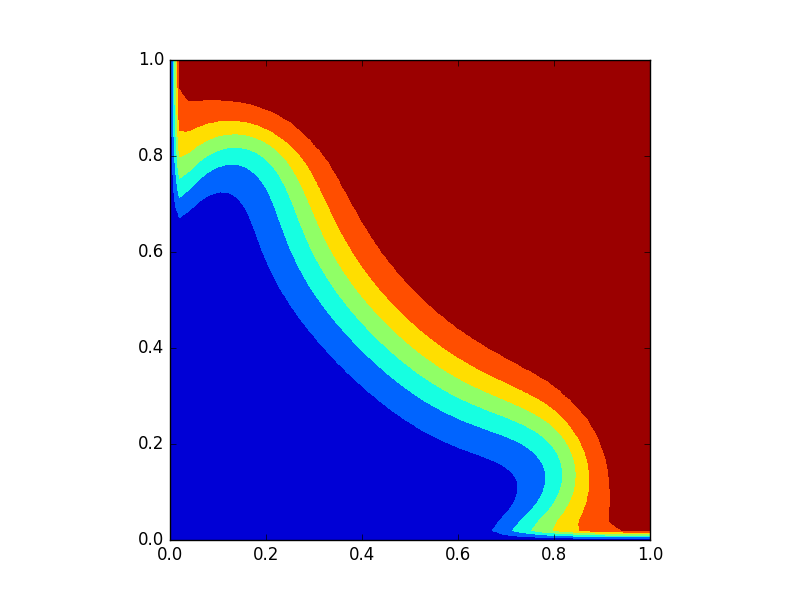}\includegraphics{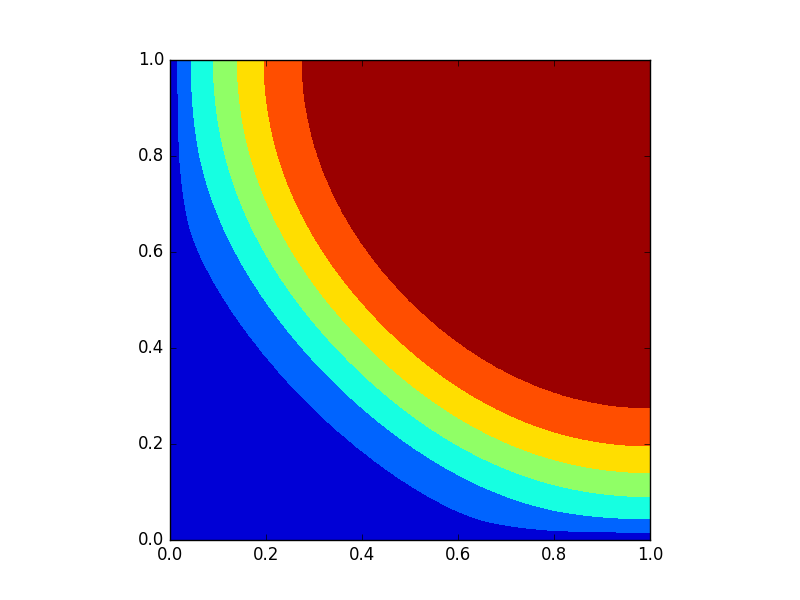}}
\resizebox{0.6\textwidth}{!}{\includegraphics{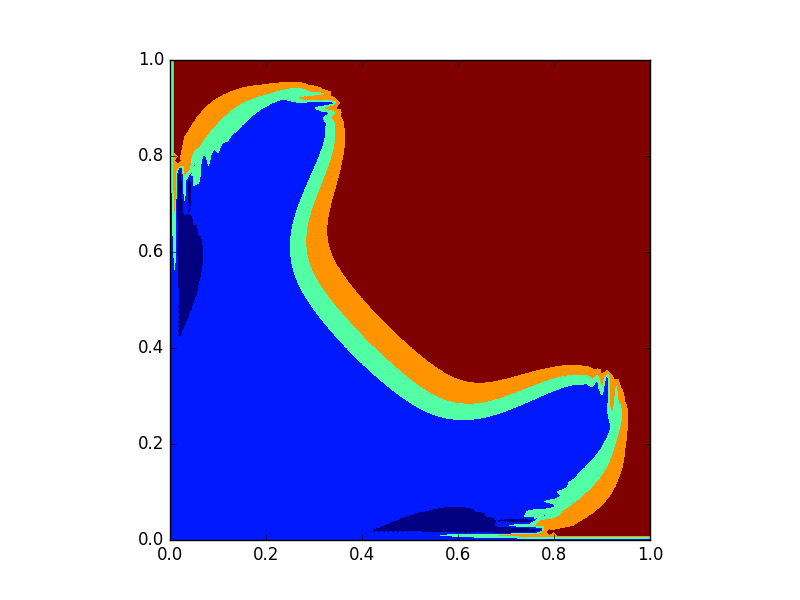}\includegraphics{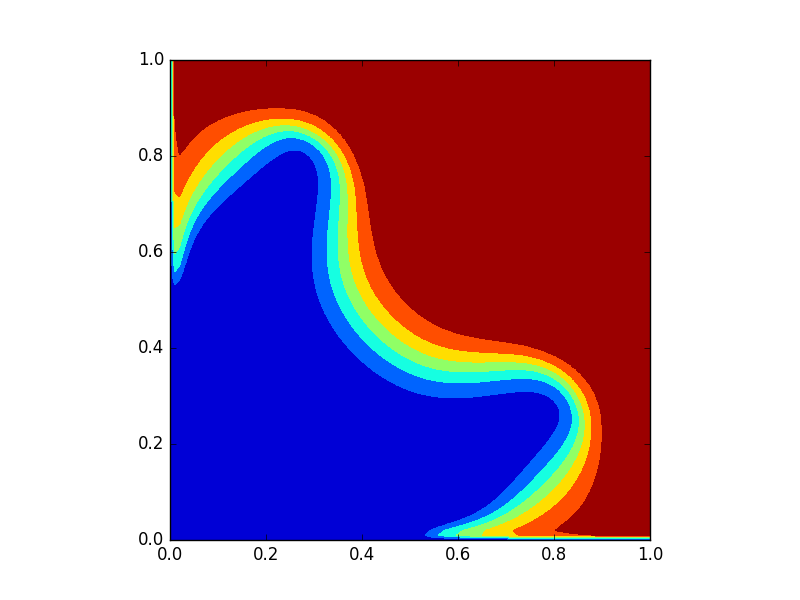}\includegraphics{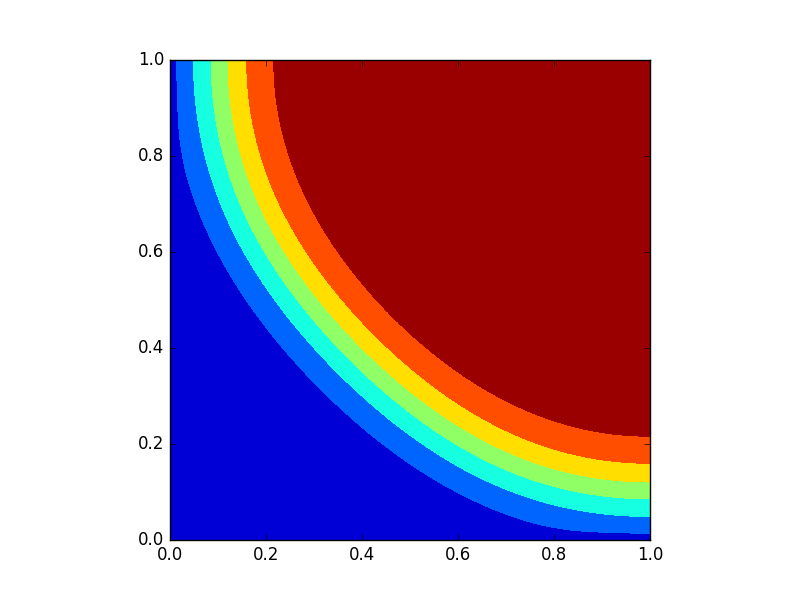}}
\end{center}
\caption{Analytical test 2, Scheme A. From top to bottom:  $25\times 25$ mesh and time step $0.02$; $50\times 50$ mesh and time step $0.01$; $100\times100$ mesh and time step $0.005$. From left to right: centred, upstream, modification by \eqref{eq:ddtmod} (right).
Values from $0$ (dark blue) to $1$ (dark red).}
\protect{\label{fig:scheme_a_ta2}}
\end{figure}

\begin{table}
\begin{tabular}{|c|c|c|c|c|}
\hline
Variant & Mesh & $\delta t$  & $L^1$ error & $L^2$ error \\
\hline
\multirow{3}{*}{centred} & $25\times 25$ & 0.02 & 1.06E-1 & 2.00E-1\\
\cline{2-5}
 & $50\times 50$ &0.01 & 1.32E-1 & 2.77E-1\\
\cline{2-5}
 & $100\times 100$ &0.005 & 1.71E-1 & 3.49E-1\\
\hline
\hline
\multirow{3}{*}{upstream} &  $25\times 25$ &0.02 & 1.51E-1 & 2.04E-1\\
\cline{2-5}
&$50\times 50$ &0.01 & 1.35E-1 & 2.65E-1\\
\cline{2-5}
&$100\times 100$ &0.005 & 2.00E-1 & 3.78E-1 \\
\hline
\hline
\multirow{3}{*}{With \eqref{eq:ddtmod}} & $25\times 25$ & 0.02 & 1.51E-1 & 2.04E-1\\
\cline{2-5}
&$50\times 50$ &0.01 & 1.11E-1 & 1.66E-1\\
\cline{2-5}
&$100\times 100$ & 0.005 & 7.80E-2 & 1.32E-1\\
\hline
\end{tabular}
\caption{Analytical test 2, errors with three variants of Scheme A: centred, upstream
and with \eqref{eq:ddtmod}}
\label{tab:ta2_a}
\end{table}

\begin{figure}[!ht] 
\begin{center}
\resizebox{0.6\textwidth}{!}{\includegraphics{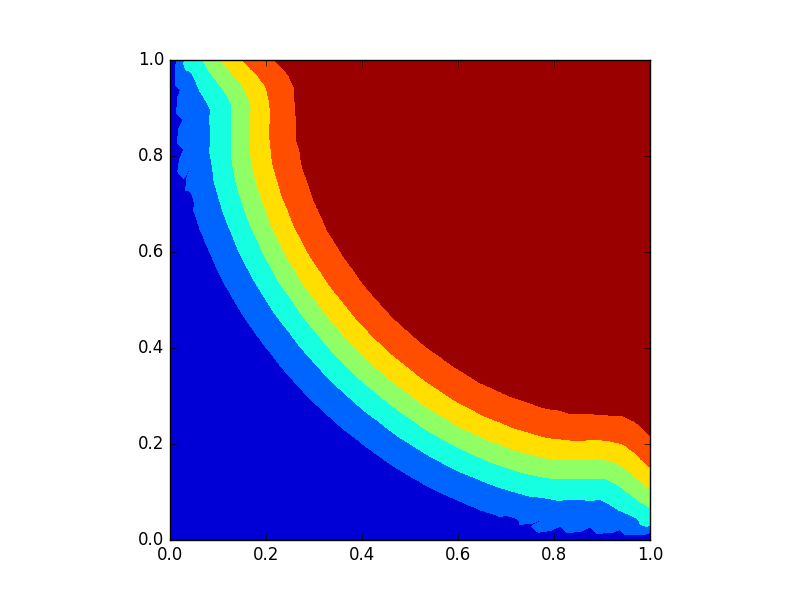}\includegraphics{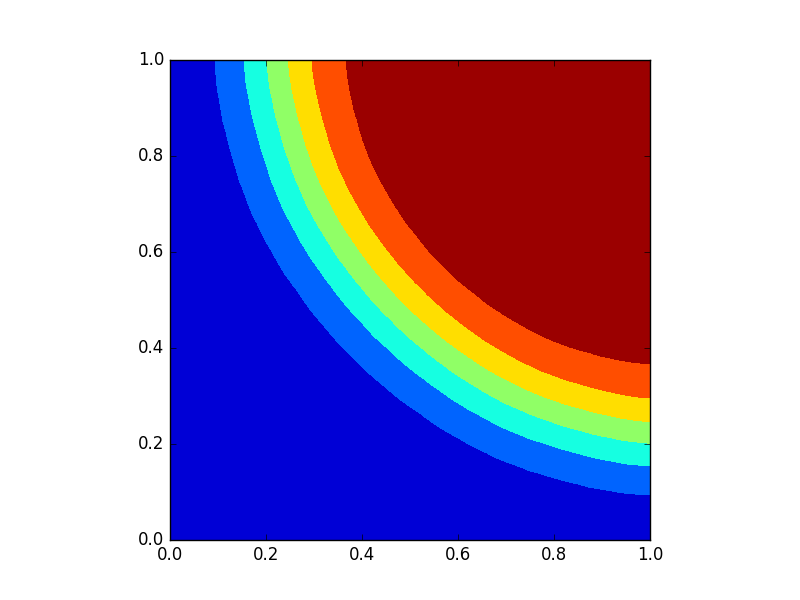}\includegraphics{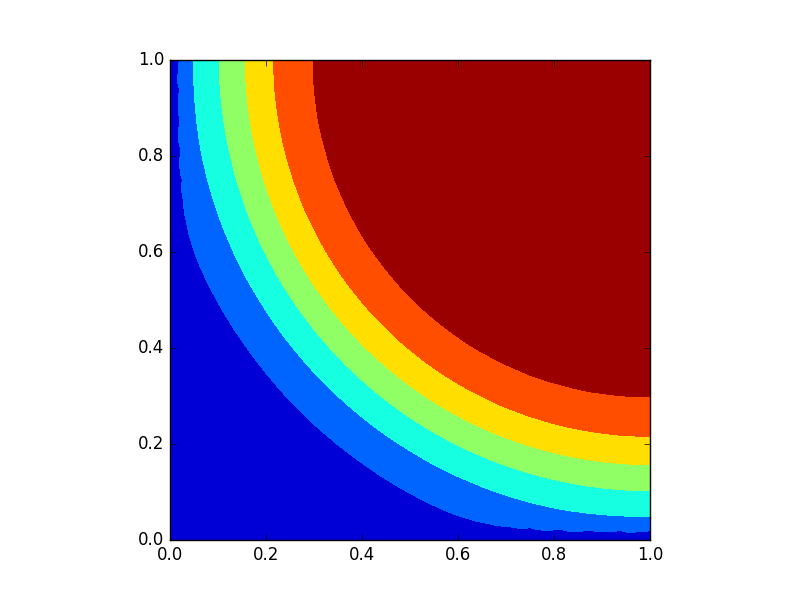}}
\resizebox{0.6\textwidth}{!}{\includegraphics{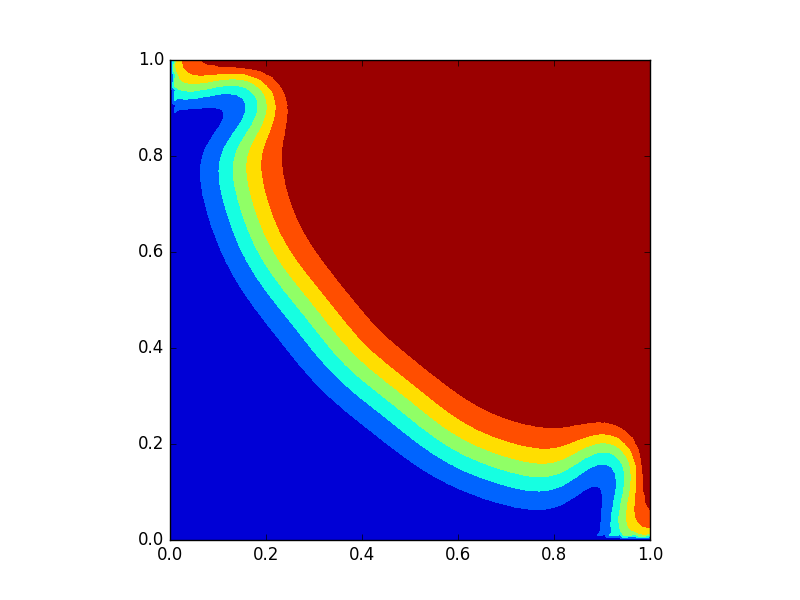}\includegraphics{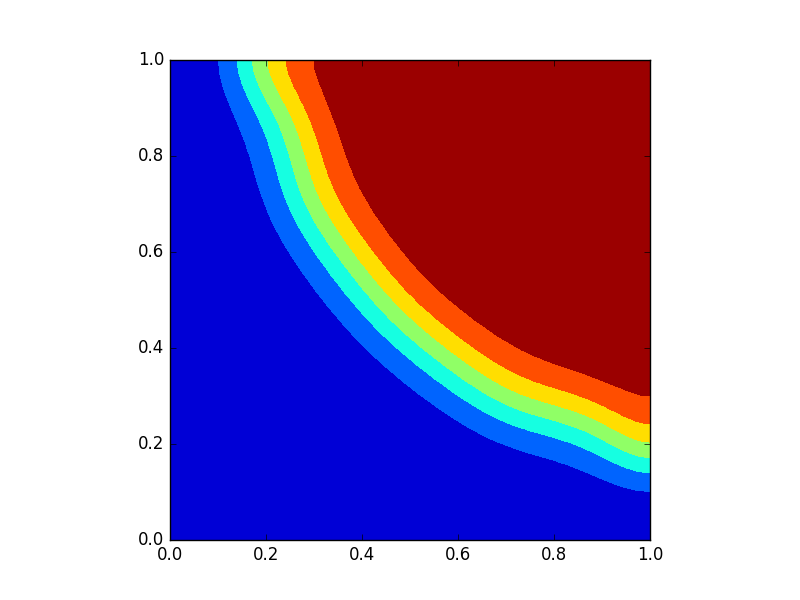}\includegraphics{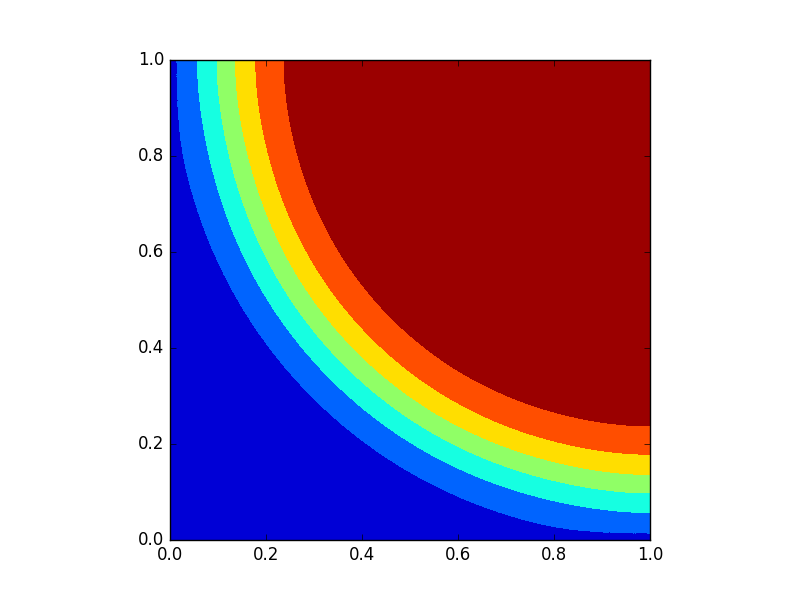}}
\resizebox{0.6\textwidth}{!}{\includegraphics{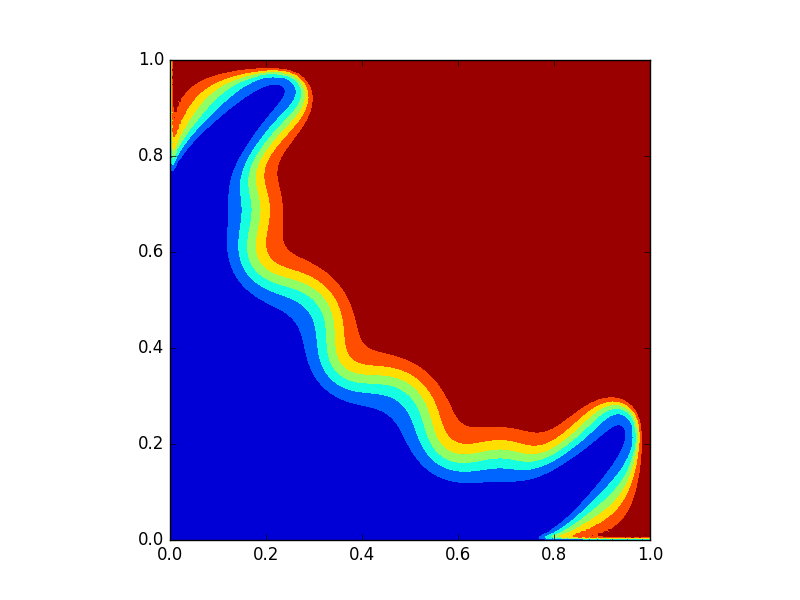}\includegraphics{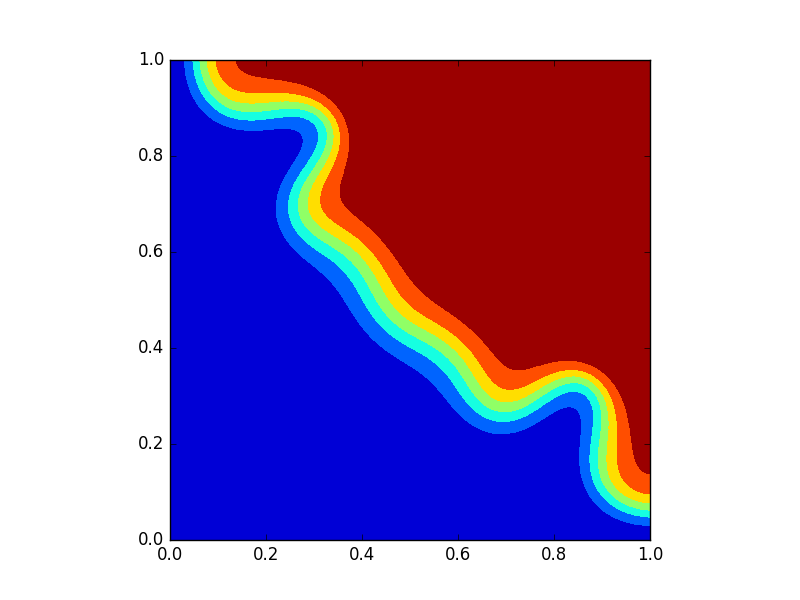}\includegraphics{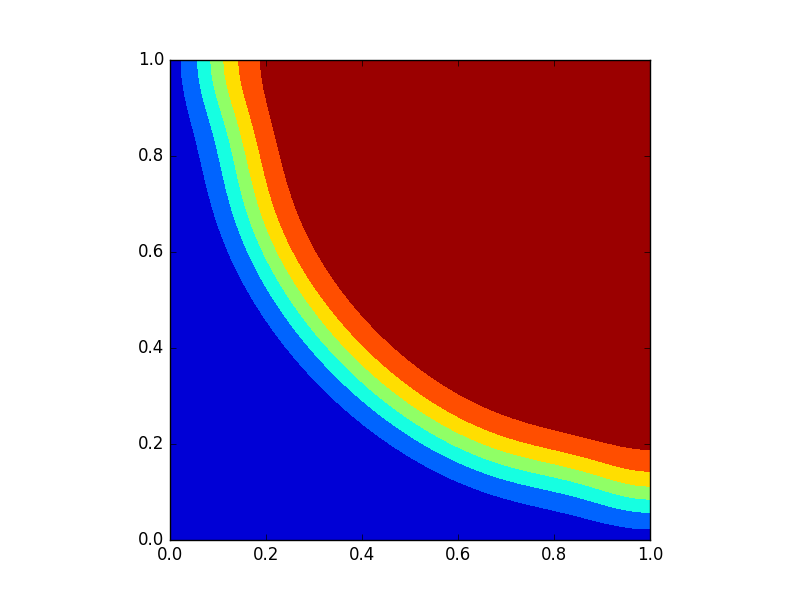}}
\end{center}
\caption{Analytical test 2, Scheme B. From top to bottom:  Mesh4 and time step $0.02$; Mesh5 and time step $0.01$; Mesh6 and time step $0.005$. From left to right: centred, upstream,   modification by \eqref{eq:ddtmod} (right).
Values from $0$ (dark blue) to $1$ (dark red).}
\protect{\label{fig:scheme_b_ta2}}
\end{figure}

\begin{table}
\begin{tabular}{|c|c|c|c|c|}
\hline
Variant & Mesh & $\delta t$  & $L^1$ error & $L^2$ error \\
\hline
\multirow{3}{*}{centred} & Mesh4 & 0.02 & 9.05E-2 & 1.46E-1 \\
\cline{2-5}
& Mesh5 & 0.01 & 7.23E-2 & 1.47E-1 \\
\cline{2-5}
& Mesh6 &0.005 &  9.42E-2 & 2.29E-1 \\
\hline
\hline
\multirow{3}{*}{upstream} & Mesh4 & 0.02 & 1.71E-1 & 2.98E-1 \\
\cline{2-5}
& Mesh5 & 0.01 & 1.58E-1 & 3.03E-1 \\
\cline{2-5}
& Mesh6 &0.005 & 1.66E-1 & 3.34E-1 \\
\hline
\hline
\multirow{3}{*}{With \eqref{eq:ddtmod}} & Mesh4 &0.02 & 1.22E-1 & 1.77E-1 \\
\cline{2-5}
&Mesh5 & 0.01 & 8.55E-2 & 1.39E-1 \\
\cline{2-5}
&Mesh6 & 0.005 & 5.65E-2 & 1.04E-1\\
\hline
\end{tabular}
\caption{Analytical test 2, errors with three variants of Scheme B: centred, upstream
and with \eqref{eq:ddtmod}}
\label{tab:ta2_b}
\end{table}

\subsection{Comparison with test cases from the literature}

We now apply Schemes A and B defined in Section \ref{ssec:examples} to  the first and second test cases in Wang-et-al. \cite{wan00} and Chainais-Hillairet--Droniou \cite{cd07}. In these test cases, the source terms do not satisfy Hypotheses \eqref{hyp:source}, but, as we show below and as already noted in \cite{cd07}, this does not seem to prevent numerical
convergence. Additionally, in the second test case, the diffusion--dispersion tensor $\DD$ does not 
satisfy all the hypotheses in \eqref{hyp:D} (it is only positive, not uniformly positive-definite). This will force us, as in the analytical test 2, to introduce some additional vanishing diffusion (see Section \ref{sec:ddnotboundedbybelow}).

In both cases the domain is a two-dimensional reservoir $\O = (0,1000)^2$ ft$^2$ and the final time is
set $T = 1080$ days ($\approx$ 3 years). Denoting by $\delta_A$ 
the Dirac measure at a point $A\in\overline{\O}$, set $\hat c(x) = 1$,  
$q^I(x)\ud x= 30\ \delta_{(1000,1000)}$ and $q^P(x)\ud x= 30\ \delta_{(0,0)}$; hence, solvent is injected at the top right corner of the domain and a mixture of oil
and solvent is recovered at the bottom left corner, both at a rate of 
$30$ ft$^2$/day. 
Set $\Phi(x) = 0.1$, $c_0(x)=0$ and, for almost-every $x\in\O$, $\K(x)=k\IdM$ with $k=80$ md. 

\begin{remark}[Implementation of the Dirac measures]
For both schemes, the reconstructions $\Pi_\disc$ provide piecewise constant
functions on a certain mesh. In the case of Scheme A, the boundary cells of this mesh are cut in 2 at the boundary edges and in four at the corners, so that the centers of the cells are located on the boundary. In the case of Scheme B, these cells are dual cells centred on the vertices of the triangular mesh. In both cases, there are four cells whose centre is precisely located at a corner of $\Omega$. In our test cases, the Dirac masses $q^+$ and $q^-$ in \eqref{scheme_p}--\eqref{scheme_c}  are simply taken into account in the source terms corresponding to the two corner cells where they are located.
\end{remark}

\subsubsection{Test 1: constant viscosity, only molecular diffusion}

We consider Test 1 of \cite{wan00}. Thus, $M=1$ and the diffusion--dispersion tensor is defined by \eqref{eq:ddt} with 
$\Phi\dm =1$ ft$^2$/day, $\Phi\dl = 0$ ft and $\Phi\dt= 0$ ft. $\DD$ then satisfies the coercivity properties that ensures that the reconstructed gradient of the concentration remains bounded. We observe a clear numerical convergence, comparing the refinement in space and time of both Schemes A and B (see Figure \ref{fig:scheme_a_b_t1}). 
\begin{figure}[!ht] 
\begin{center}
\resizebox{0.6\textwidth}{!}{\includegraphics{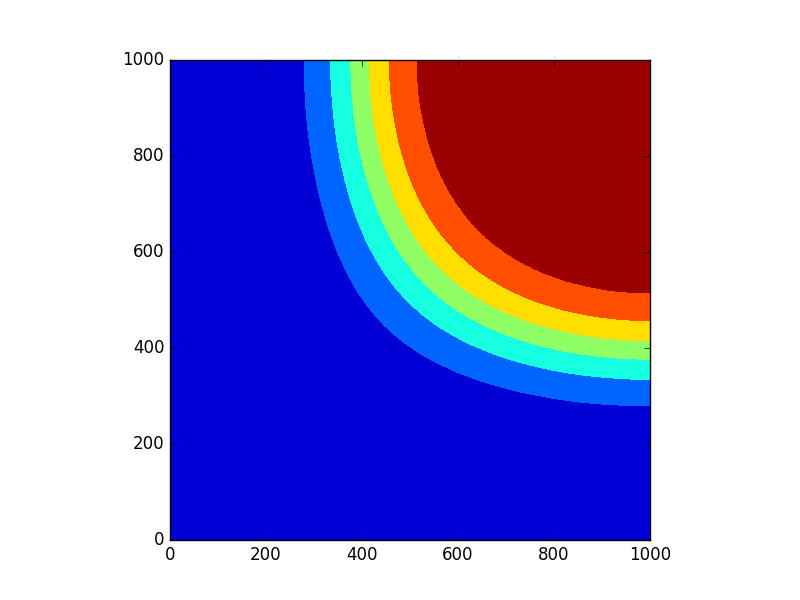}\includegraphics{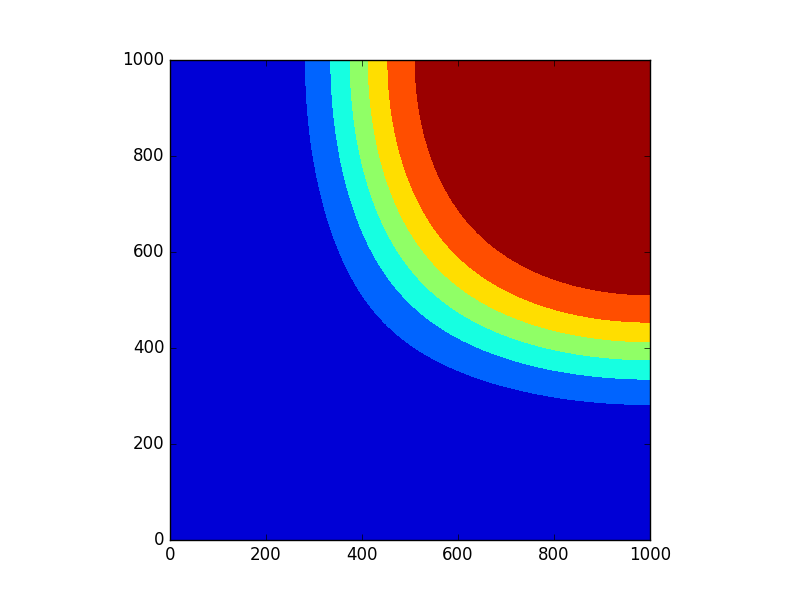}}
\resizebox{0.6\textwidth}{!}{\includegraphics{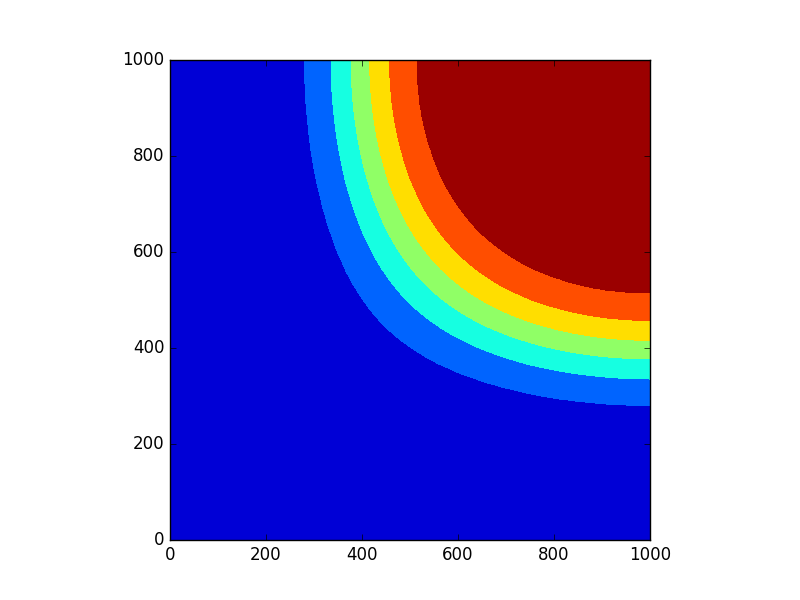}\includegraphics{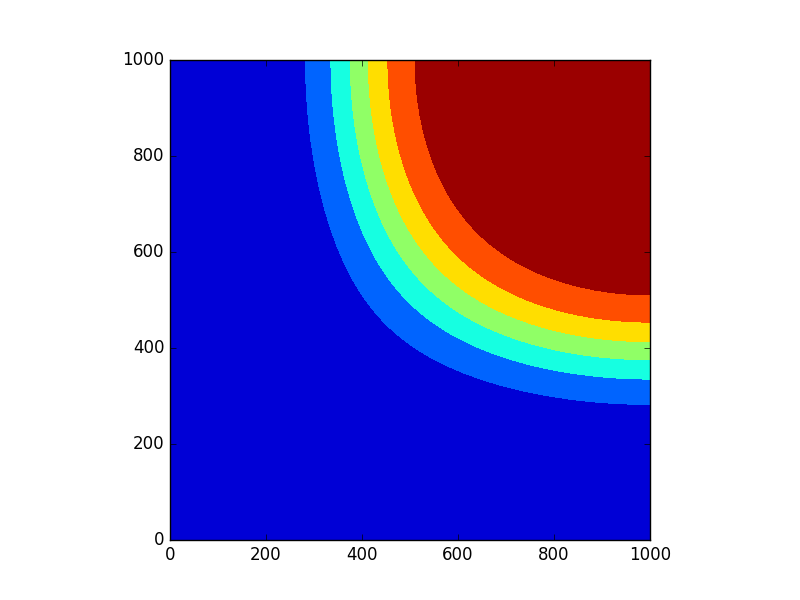}}
\end{center}
\caption{Test 1, concentration with centred scheme. Top: Scheme A :  $50\times50$ mesh with time step $18$ days (left); $100\times100$ mesh with  time step $9$ days (right).
Bottom: Scheme B on Mesh4 with time step $18$ days (left); on Mesh5 with time step $9$ days (right). 
Values from $0$ (dark blue) to $1$ (dark red).}
\protect{\label{fig:scheme_a_b_t1}}
\end{figure}

\subsubsection{Test 2: transverse and longitudinal dispersion, no molecular diffusion}\label{sec:ddnotboundedbybelow}

We now consider Test 2 of \cite{wan00}, in which $M=41$,
$\Phi\dm =0$ ft$^2$/day, $\Phi\dl = 5$ ft and $\Phi\dt= 0.5$ ft. Let us first examine the results provided by Schemes A and B used without modification. 
\begin{figure}[!ht] 
\begin{center}
\resizebox{0.9\textwidth}{!}{\includegraphics{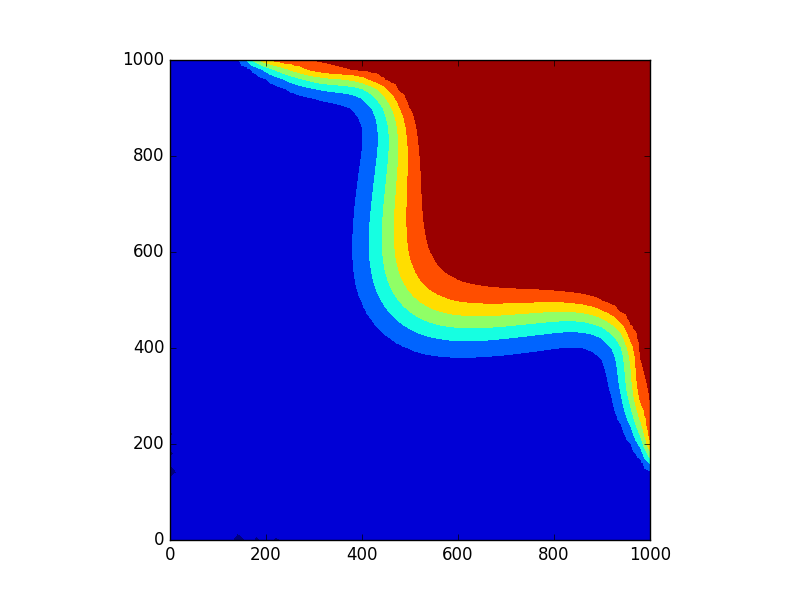}\includegraphics{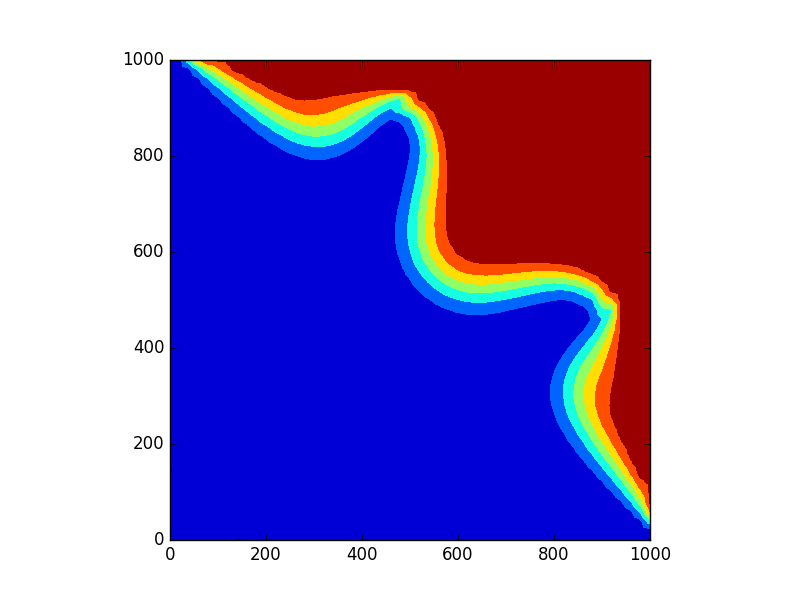}\includegraphics{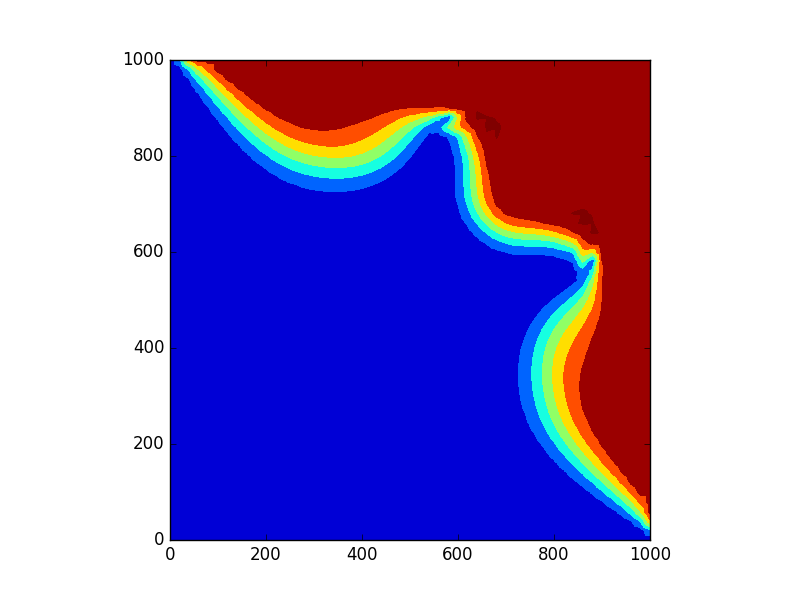}}
\resizebox{0.9\textwidth}{!}{\includegraphics{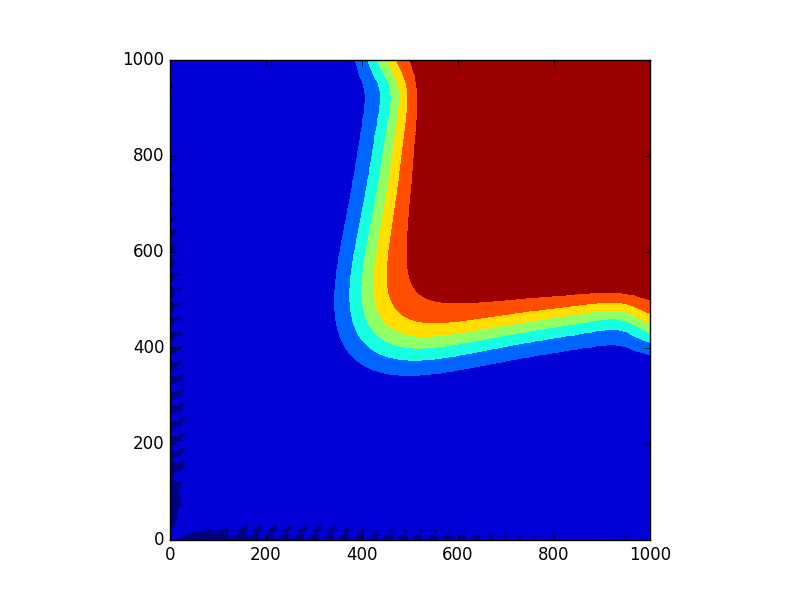}\includegraphics{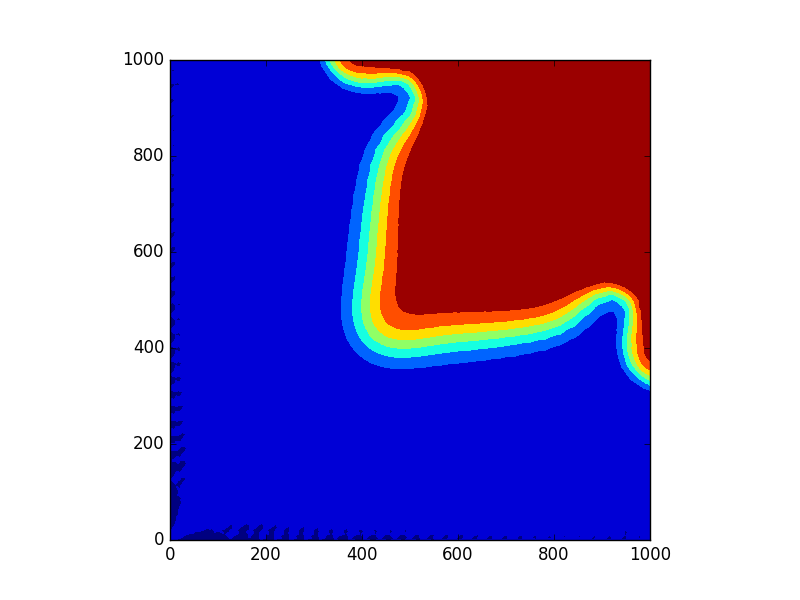}\includegraphics{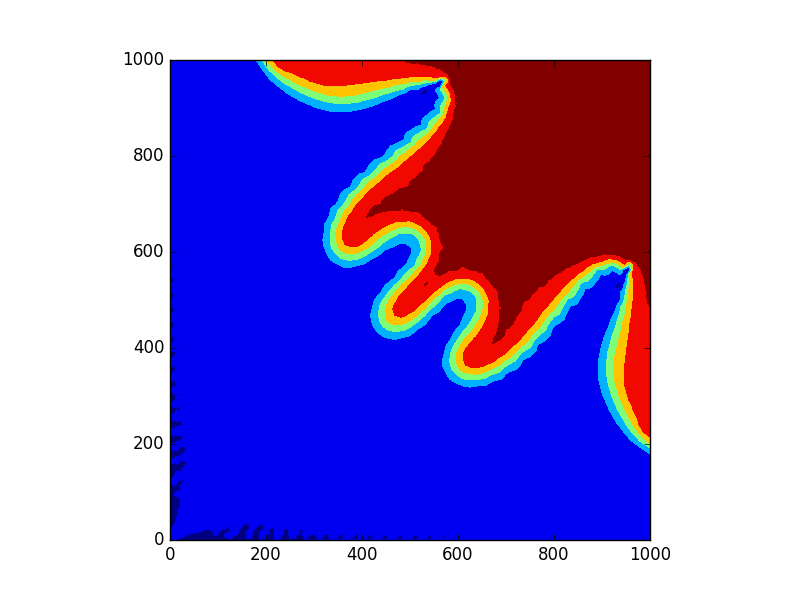}}
\end{center}
\caption{Top: Concentration with scheme A, $50\times50$ mesh with time step $36$ days (left), $18$ days (middle), $9$ days (right). 
Bottom: same with scheme B on Mesh5. 
Values from $0$ (dark blue) to $1$ (dark red).}
\protect{\label{fig:scheme_center}}
\end{figure}
Figure \ref{fig:scheme_center} shows the contours of the approximate concentration obtained at 
the final time $T$. We notice that many published results consider time steps equal to or larger than 36 days; with such a time step, the result obtained with scheme B is comparable to the ones in the literature. 
We notice however that, for smaller time steps, the results are no longer so nice. In that case, the transversal diffusion is not sufficient to stabilize the scheme, and introducing a vanishing diffusion becomes necessary.

\begin{remark}The reason for the better results observed with a larger time step can perhaps be found in the term $\frac{1}{2}\left(\PiD c^{(n+1)} - \PiD c^{(n)}\right)^2$ in \eqref{est:parabolic_disct}. This term can be seen as an approximation of $\delta t\int_0^T\int_{\Omega} (\partial_t c)^2$. When carried over in \eqref{eq:2ndenergy}, it therefore contributes to controling this time derivative of the concentration, which can result in an improved stability of the scheme. For smaller time steps, this control becomes less efficient due to the $\delta t$ factor.
\end{remark}

We saw in the analytical test 2 that upstreaming the schemes was not necessarily a good option to recover a stable and accurate solution. This is confirmed here in Figure \ref{fig:scheme_upstream}:
upstreaming is not a good paliative for the lack of diffusion (especially for Scheme A).
\begin{figure}[!ht] 
\begin{center}
\resizebox{0.9\textwidth}{!}{\includegraphics{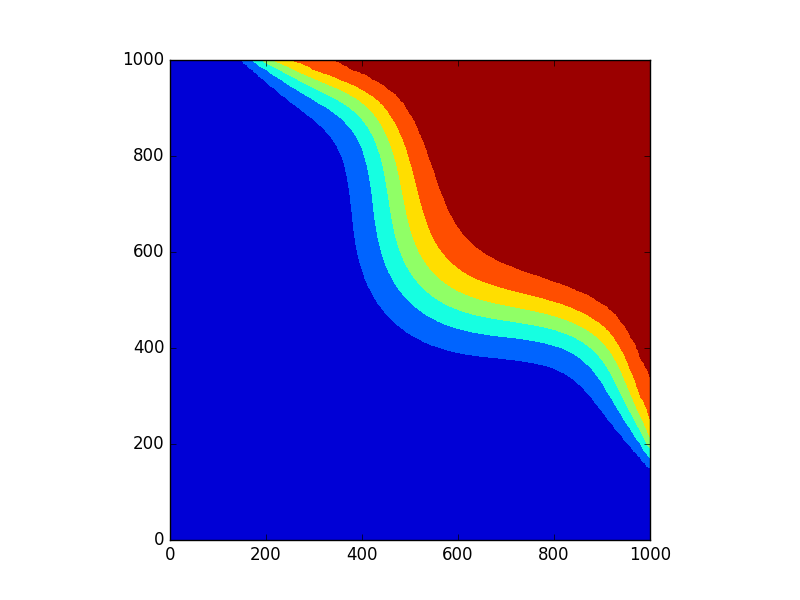}\includegraphics{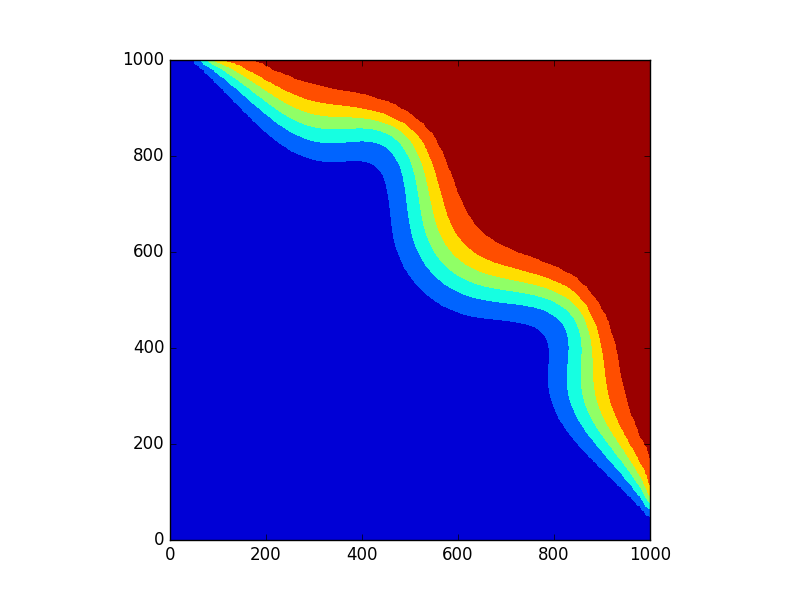}\includegraphics{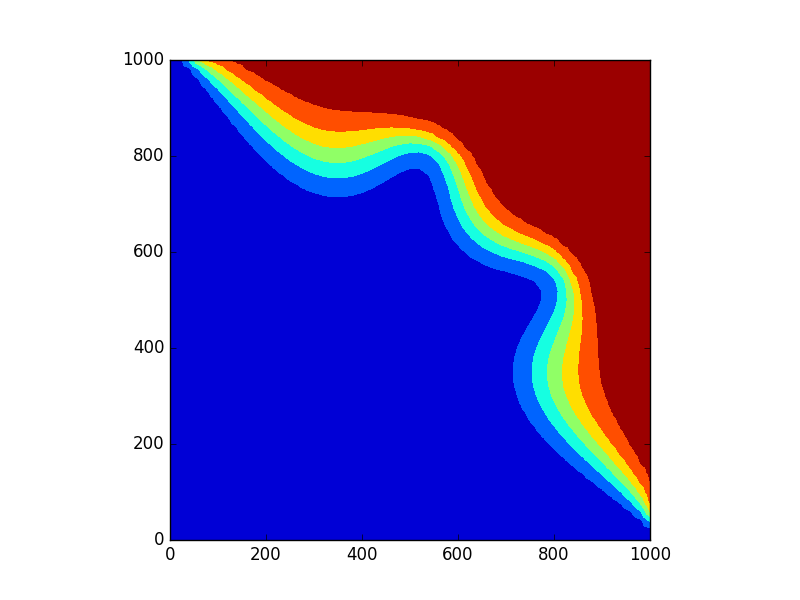}}
\resizebox{0.9\textwidth}{!}{\includegraphics{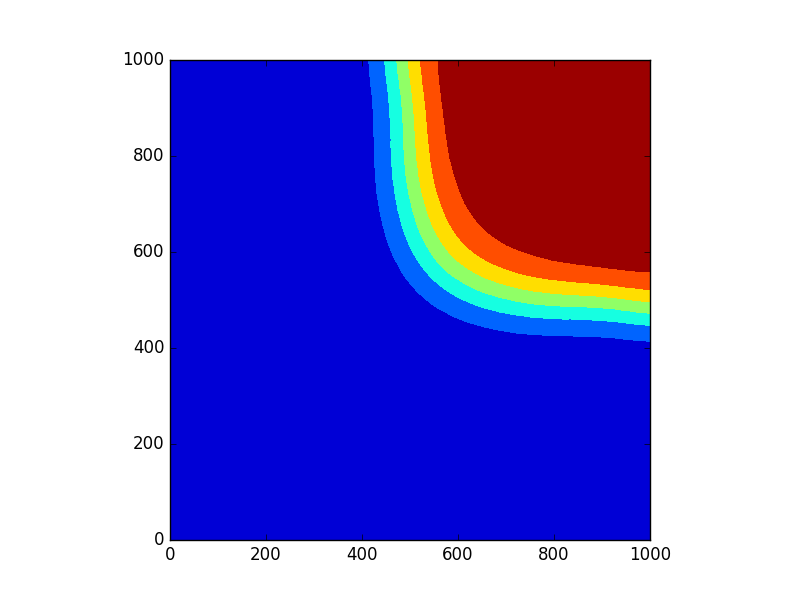}\includegraphics{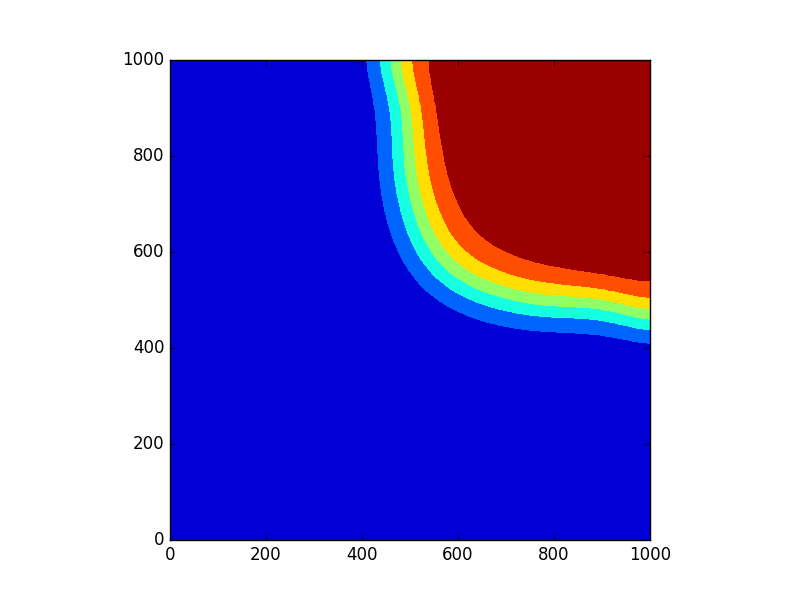}\includegraphics{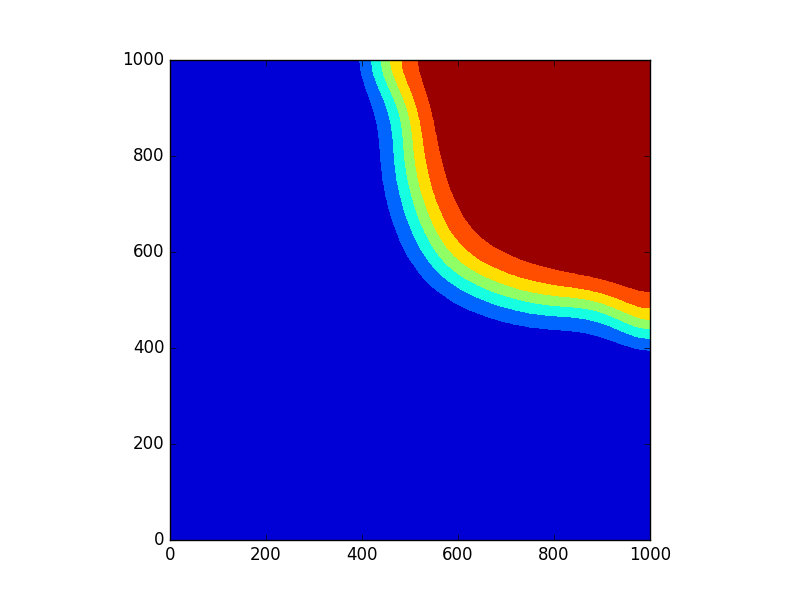}}
\end{center}
\caption{Top: Concentration with upstream version of scheme A, $50\times50$ mesh with time step $36$ days (left), $18$ days (middle), $9$ days (right). 
Bottom: same with upstream  version of scheme B on Mesh5. 
Values from $0$ (dark blue) to $1$ (dark red).}
\protect{\label{fig:scheme_upstream}}
\end{figure}

To recover a stable and accurate solution, we use the modification \eqref{eq:ddtmod}.
The results obtained with this modified scheme are shown in Figure \ref{fig:scheme_h}. As can be seen, both Schemes A and B then display a nice numerical convergence to the expected solution, as the time step decreases.

It was proved in \cite{dt16} that, when the molecular diffusion vanishes, weak solutions to \eqref{peacemanmodel} converge to solutions of the same model with zero molecular diffusion; the arguments given in \cite{dt16} are applicable to ${\DD}_h$ defined by \eqref{eq:ddtmod} and show that, as $h\to 0$, the solution with this modified diffusion--dispersion tensor converges to a solution of the model with zero molecular diffusion. This justifies using \eqref{eq:ddtmod} in numerical approximations.
\begin{figure}[!ht] 
\begin{center}
\resizebox{0.9\textwidth}{!}{\includegraphics{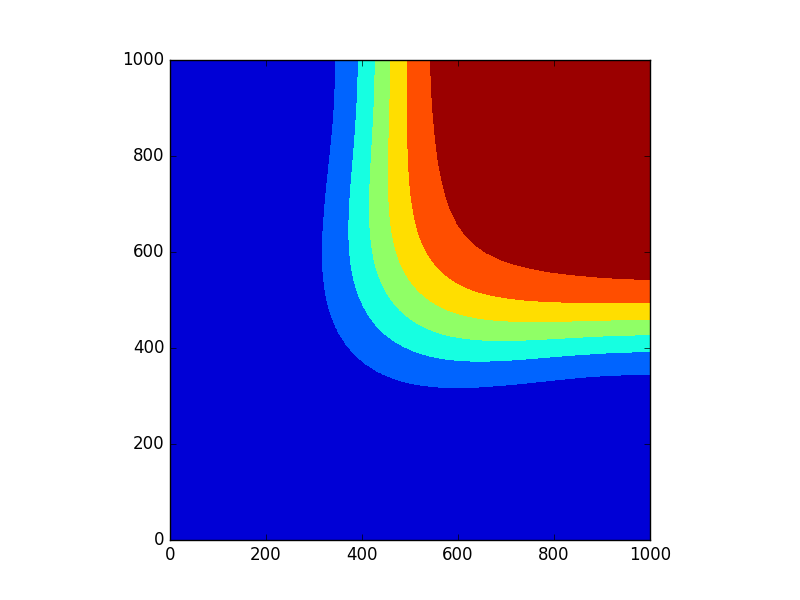}\includegraphics{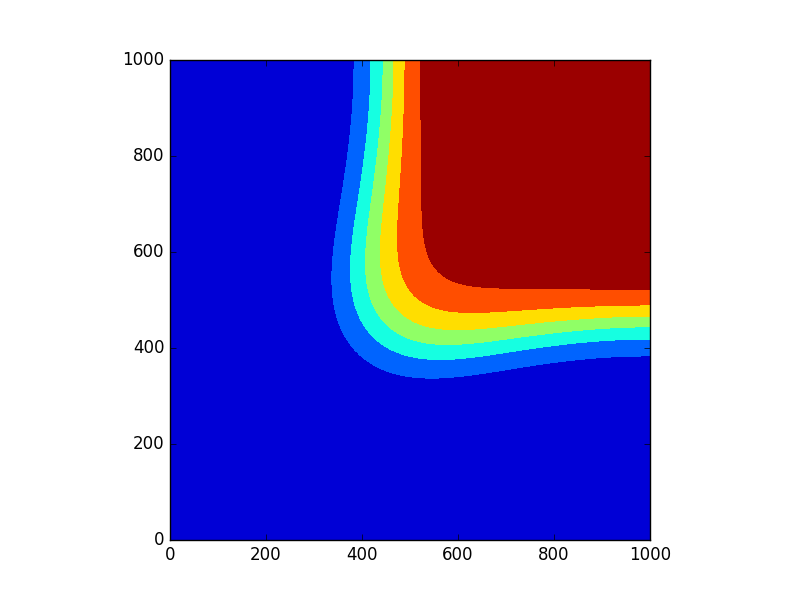}\includegraphics{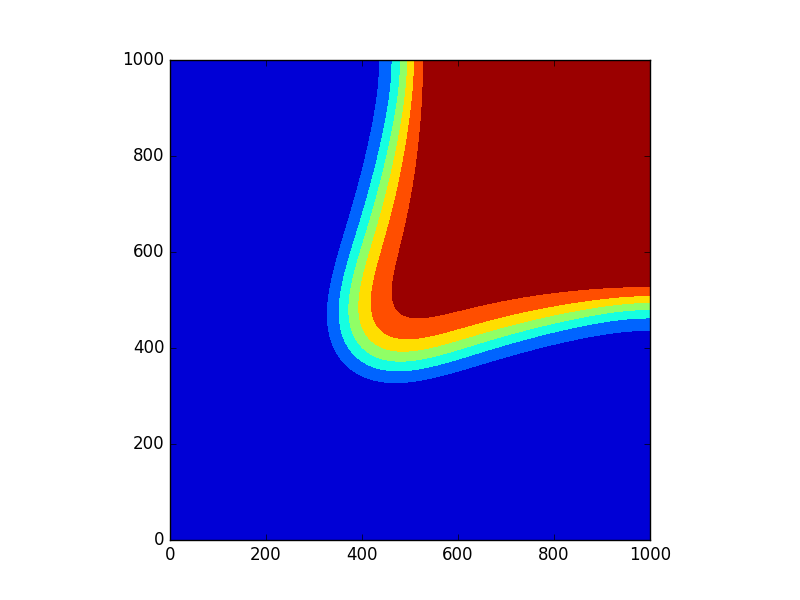}}
\resizebox{0.9\textwidth}{!}{\includegraphics{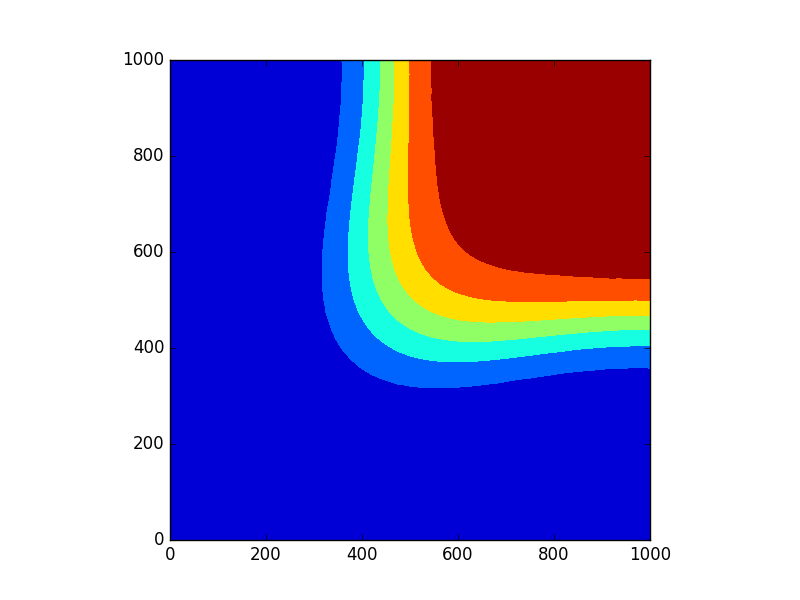}\includegraphics{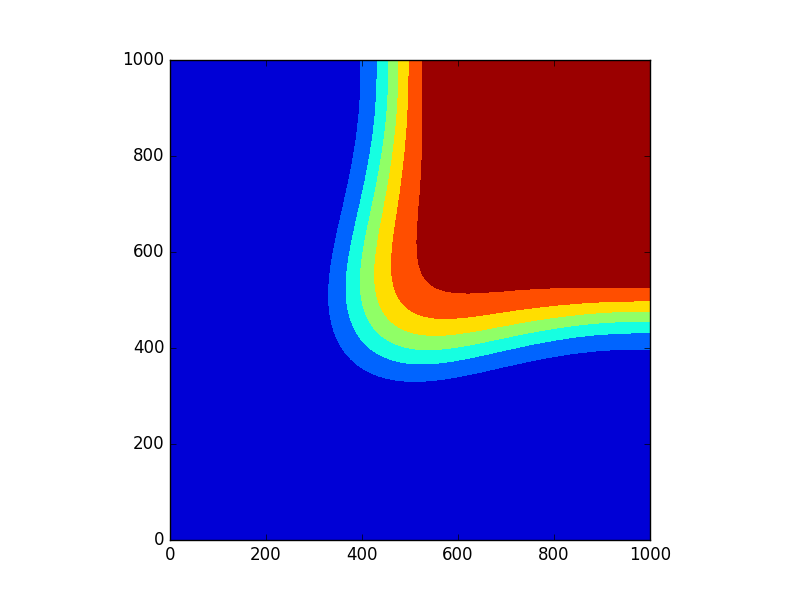}\includegraphics{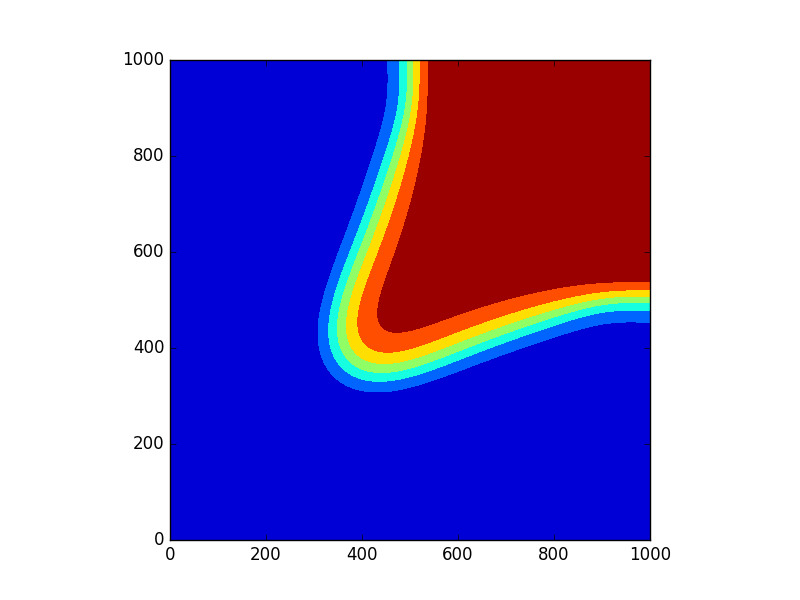}}
\end{center}
\caption{Top: Concentration using \eqref{eq:ddtmod} with Scheme A: $50\times50$ mesh with time step $36$ days (left), $100\times100$ mesh with time step $18$ days (middle), $200\times200$ mesh with time step $9$ days (right).
Bottom: same with Scheme B, Mesh4  with time step $36$ days (left), Mesh5  with time step $18$ days (middle), Mesh6 with time step $9$ days (right).
Values from $0$ (dark blue) to $1$ (dark red).}
\protect{\label{fig:scheme_h}}
\end{figure}

\section{Conclusion}

We applied the gradient discretisation method to a model of miscible incompressible flows in porous
media. The GDM framework enables us to write in a unified format many different numerical methods
for this model, from finite differences, to finite volumes and finite elements. We considered
a centred discretisation of the advective terms in the concentration equation.
The convergence analysis was performed using compactness techniques, to avoid imposing non-physical regularity assumptions on the data or the solution to the model. It applies to all methods fitting into the GDM framework.
A novelty of our analysis compared to similar results in the literature
is a uniform-in-time, $L^2(\O)$-strong convergence result for the approximate concentration.

We showed numerical results using two schemes that fit into the GDM framework: a finite-difference scheme
written on Cartesian meshes, and a mass-lumped $P^1$ finite element scheme on triangles. {It was demonstrated, on both analytical and physical test cases, that centred and upstream schemes behave badly in the case of varying viscosity and small physical diffusion. A modification was then proposed and shown to lead to stable and accurate solutions. This modification consists introducing some vanishing numerical diffusion, designed to be isotropic, to scale as the magnitude of the Darcy velocity, and to vanish with the mesh size in the same was as upstream numerical diffusions.}

% --- APPENDIX ---
\appendix

\section{Convergence lemmas}

For proofs of the first three of these lemmas, see \cite{dt14}. Lemma \ref{lem:equiv-unifconv}
is proved in \cite{DE15}, and the interpolation lemma \ref{lem:interp} is a special case of
\cite[Lemma 4.10]{gdmbook}.

% space h lemma
\begin{lemma} \label{lem:h}
Let $\O$ be a bounded subset of $\RR^{N}$, $N\in\NN$ and for each $n\in\NN$, 
let $H_{n}:\Omega\times\RR^{N}\to\RR$ be a Carath\'eodory function such that
\begin{itemize}
\item there exist positive constants $C$ and $\gamma$ such that for almost-every $x\in\O$, 
\begin{equation}
|H_{n}(x,\xi)|\leq C(1+|\xi|^\gamma),\quad\mbox{$\forall\xi\in\RR^N$, $\forall n\in\NN$,} \label{eq:lem1}
\end{equation}
\item there is a Carath\'eodory function $H:\O\times\RR^N\to\RR$ such that for a.e. $x\in\O$, 
\begin{equation}
H_n(x,\cdot)\to H(x,\cdot)\quad\mbox{uniformly on compact sets as $n\to\infty$.} \label{eq:lem2}
\end{equation} 
\end{itemize}
If $p\in[\max(1,\gamma),\infty)$ and $(u_n)_{n\in\NN}\subset L^p(\O)^N$ is a sequence with 
$u_n\to u$ in $L^p(\O)^N$ as $n\to\infty$, then
$H_n(\cdot,u_n)\to H(\cdot,u)$ in $L^{p/\gamma}(\Omega)$ as $n\to\infty$.
\end{lemma}

% spacetime h lemma
\begin{corollary} \label{cor:h}
Let $\O$, $H_n$ and $H$ satisfy the hypotheses of Lemma \ref{lem:h}. If $p,q\in [\max(1,\gamma),\infty)$ 
and $(u_n)_{n\in\NN}\subset L^p(0,T; L^q(\O)^N)$ is a sequence with 
$u_n\to u$ in $L^p(0,T; L^q(\O)^N)$ as $n\to\infty$, then
$H_n(\cdot,u_n)\to H(\cdot,u)$ in $\Leb{p/\gamma}{q/\gamma}$ as $n\to\infty$.
\end{corollary}

% weak-strong lemma
\begin{lemma} \label{lem:ws}
Let $\O$ be a bounded, open subset of $\RR^{N}$ and for every $n\in\NN$, 
let $w_n:\O\times(0,T)\to\RR$ and $v_n:\O\times(0,T)\to\RR$ be such that 
$w_n\to w$ in $\Leb{r_1}{s_1}$, and
$v_n\weakto v$ weakly in $\Leb{r_2}{s_2}$,
where $r_1,r_2,s_1,s_2\geq1$ are such that $1/r_{1} + 1/r_{2}\leq 1$ 
and $1/s_{1} + 1/s_{2}\leq 1$. Suppose also that the sequence $(w_{n}v_{n})_{n\in\NN}$ 
is bounded in $\Leb{a}{b}$, where $a,b\in(1,\infty)$. Then
$w_{n}v_{n}\weakto wv$ weakly in $\Leb{a}{b}$.
\end{lemma}

% equivalence condition for uniform convergence
\begin{lemma} \label{lem:equiv-unifconv}
Let $(K,d_K)$ be a compact metric space, $(E,d_E)$ a metric space. 
Denote by $\cF(K,E)$ the space of functions $K\to E$, endowed with
the uniform metric $d_\cF(v,w)=\sup_{s\in K}d_E(v(s),w(s))$ (note that
this metric may take infinite values).

Let $(v_n)_{n\in\NN}$ be a sequence in $\cF(K,E)$ and $v:K\to E$ be
continuous. Then $v_n\to v$ for $d_{\cF}$ if and only if, for any $s\in K$ and
any sequence $(s_n)_{n\in\NN}\subset K$ converging to $s$ for $d_K$,
$v_n(s_n)\to v(s)$ for $d_E$.
\end{lemma}

% Interpolation lemma
\begin{lemma}[GDM interpolation of space-time functions]\label{lem:interp}
Let $(\discm)_{m\in\NN}$ be a consistent sequence of gradient discretisations in the sense
of Definition \ref{def:consistency}. Let $\theta\in\{0,1\}$ and, if $v\in \unknownsm$, set, for all $n=0,\ldots,N_m-1$
and $t\in (t^{(n)},t^{(n+1)}]$
\[
\PiDm^{(\theta)}v(t)=\PiDm v^{(n+\theta)}\quad\mbox{ and }\quad\gradDm^{(\theta)}(t)=\gradDm v^{(n+\theta)}
\]
(hence, using the notations \eqref{imp.ex}, $\PiDm^{(1)}=\PiDm$ and $\PiDm^{(0)}=\PiDme$).
Let $\varphi\in L^2(0,T;H^1(\O))$. Then there exists $v_m=(v_m^{(n)})_{n=0,\ldots,N_m}\in \unknownsm^{N_m+1}$
such that, as $m\to\infty$,
\[
\PiDm^{(\theta)}v_m\to \varphi\mbox{ in $\Leb{2}{2}$}\quad\mbox{ and }\quad\gradDm^{(\theta)}v_m\to \nabla\varphi\mbox{ in $\Leb[d]{2}{2}$.}
\]
If, moreover, $\partial_t\varphi\in \Leb{2}{2}$ then $(v_m)_{m\in\NN}$ can be chosen such that, additionally,
\[
\deltaDm v_m\to \partial_t\varphi\mbox{ in $\Leb{2}{2}$}\quad\mbox{ and }\quad\PiDm v_m^{(0)}\to \varphi(\cdot,0)\mbox{ in $L^2(\O)$}.
\]
\end{lemma}

% --- BIBLIOGRAPHY ---

\bibliography{peaceman-gs}{}
\bibliographystyle{plain}

\end{document}